\documentclass[]{amsart}
\usepackage{amsmath, amsfonts, amssymb, amsthm, mathrsfs, mathtools}
\usepackage[a4paper, total={6in, 8in}]{geometry}
\usepackage{hyperref}
\usepackage{xcolor}
\usepackage{soul}
\usepackage{graphicx}
\usepackage{caption,subcaption,comment}
\usepackage{tikz}
\usepackage{pgfplots}
\usepackage{adjustbox}
\usepackage{tabularx}

\allowdisplaybreaks 

\hypersetup{
    colorlinks,
    linkcolor={blue!80},
    citecolor={blue!80},
    urlcolor={blue!80}
}

\usepackage{lineno}
\usepackage{todonotes}

\newcommand{\grado}{\nabla_{\Omega}}
\newcommand{\lapo}{\Delta_{\Omega}}
\newcommand{\gradoh}{\nabla_{\Omega_h}}
\newcommand{\matdev}{\partial^\bullet}

\newcommand{\diam}{\operatorname{diam}}
\newcommand{\esssup}[1]{\underset{#1}{\operatorname{ess\,sup}}\,}
\newcommand{\Ph}{\mathscr{P}_h}
\newcommand{\Eh}{\mathscr{E}_h}
\newcommand{\dt}{\mathrm{d}t}
\newcommand{\ddt}[1]{\frac{\mathrm{d}#1}{\mathrm{d}t}}
\newcommand{\G}{\mathcal{G}}
\newcommand{\Gh}{\mathcal{G}_h}
\newcommand{\GSh}{\mathcal{G}_{S_h}}

\newcommand{\angled}[1]{\left \langle #1 \right \rangle}

\newcommand{\mbb}[1]{\mathbb{#1}}
\newcommand{\mbf}[1]{\mathbf{#1}}
\newcommand{\codeletter}[1]{\boldsymbol{\mathsf{#1}}}

\newcommand{\Xint}[1]{\mathchoice
	{\XXint\displaystyle\textstyle{#1}}%
	{\XXint\textstyle\scriptstyle{#1}}%
	{\XXint\scriptstyle\scriptscriptstyle{#1}}%
	{\XXint\scriptscriptstyle\scriptscriptstyle{#1}}%
	\!\int}
\newcommand{\XXint}[3]{{\setbox0=\hbox{$#1{#2#3}{\int}$ }
		\vcenter{\hbox{$#2#3$ }}\kern-.6\wd0}}

\newcommand{\dashint}{\Xint-}
\newcommand{\mval}[2]{\dashint_{#2} #1}

\pgfplotsset{compat = 1.18} 

\title[Finite element approximation of the enthalpy formulation on evolving surfaces]{Finite element approximation of the enthalpy formulation for Stefan problems on evolving surfaces}
\author[P.~J.~Herbert]{Philip~J.~Herbert}
\author[T.~Sales]{Thomas~Sales} 
\author[C.~Venkataraman]{Chandrasekhar~Venkataraman}
\address{Department of Mathematics, University of Sussex, BN1 9RF Brighton, United Kingdom}
\email{\href{mailto:p.herbert@sussex.ac.uk}{p.herbert@sussex.ac.uk}, \href{mailto:t.p.sales@sussex.ac.uk}{t.p.sales@sussex.ac.uk}, \href{mailto:c.venkataraman@sussex.ac.uk}{c.venkataraman@sussex.ac.uk}}


\newtheorem{theorem}{Theorem}[section]

\newtheorem{lemma}[theorem]{Lemma}
\newtheorem{definition}[theorem]{Definition}

\newtheorem{remark}[theorem]{Remark}
\newtheorem{assumption}{Assumption}
\newtheorem*{example}{Example}

\numberwithin{equation}{section}

\begin{document}

\begin{abstract}
	We propose, and analyse, a spatially discrete evolving surface finite element method for the approximation of the enthalpy formulation of the two-phase Stefan problem posed on an evolving surface.
    Our approach does not rely on mass-lumping and discrete maximum principles.
    We prove this numerical method is numerically stable, and prove $\mathcal{O}(\sqrt{h})$ error bounds for the temperature in the $L^2_{L^2}$ norm under minimal regularity assumptions by introducing a new projection-type operator.
    We complement our analysis with discussion on the implementation of this numerical method, where we propose a novel implementation that avoids errors due to numerical quadrature and which has not previously been considered in the literature even in the stationary, flat setting.
    We also include numerical experiments and experimental order of convergence demonstrations.
\end{abstract}

\maketitle

\section{Introduction}
\label{section: introduction}
We are interested in the analysis of evolving surface finite element methods (ESFEM) for two-phase Stefan problems posed on an evolving surface, $\Omega(t) \subset \mbb{R}^3$, moving with a prescribed velocity, $\mbf{V}$.
In particular, we consider the so-called \emph{enthalpy formulation} of the Stefan problem,
\begin{subequations}
\label{eqn: stefan problem pair}
\begin{gather}
	\matdev e + e (\grado \cdot \mbf{V}) - \lapo u = f, \text{ on } \Omega(t), \label{eqn: stefan problem}\\
	 e \in \beta(u), \label{eqn: stefan inclusion}
\end{gather}
\end{subequations}
with initial condition $e(0) = e_0 \in L^2(\Omega(0))$ and where the enthalpy $\beta$ is a set-valued map, as discussed in Section~\ref{section: stefan problem on evolving surface}.
Here the function $e$ is the enthalpy function (i.e.~the heat content) and $u$ is the temperature distribution.
We shall assume that $\Omega(t)$ is a sufficiently smooth surface without boundary, and hence there are no boundary conditions associated with~\eqref{eqn: stefan problem pair}.
We defer discussion of the differential operators used in~\eqref{eqn: stefan problem pair} until Section~\ref{section: stefan problem on evolving surface}.

Problems of the form \eqref{eqn: stefan problem pair} were proposed and analysed in~\cite{Alphonse2015Stefan}, wherein the authors study the well-posedness for suitably defined solutions.
Stefan-type problems posed on a surface appear in nature as models of phase transitions on surfaces, for example when a soap bubble freezes~\cite{Ahmadi2019SoapBubbles,Garcke2024CrystalGrowth}, in free boundary limits of bulk-surface models of ligand-receptor dynamics~\cite{Alphonse2024FreeBoundaryLimit,Elliott2017CoupledFreeBoundary}, and in models of cell-polarisation~\cite{Logioti2021CellPolarization,Niethammer2020CellPolarization}.
There is also interest in industrial applications of Stefan problems on surfaces, such as  welding~\cite{Colera2025ComparisionEnthalpy,Nedjar2002NonlinearHeat,Nehad1995KeyholePlasma} and aircraft icing~\cite{Guardone2025Aircraft,Peters2024IceCrystal},  which can be modelled as Stefan problems on surfaces (which may also be deforming).
This enthalpy formulation provides a generalised notion of solution for the Stefan problem and has been used extensively since initial work in the 1960s by Ole\u{\i}nik~\cite{Oleinik1960GeneralStefan}, Kamenomostskaja~\cite{Kamenomostskaja1961Stefan}, and Friedman~\cite{Friedman1968Stefan}.
This weak notion of solution allows one to implicitly capture the evolution of the free boundary, $\Gamma(t)$ as illustrated in Figure~\ref{fig: stefan diagram}, through the definition of the graph $\beta$ (sometimes referred to as the generalised enthalpy~\cite{Elliott1982MovingBoundaryProblems}), instead of explicitly capturing the free boundary, as in front-tracking methods~\cite{Chen1997LevelSetMethod}.
We note that sufficiently regular enthalpy solutions do indeed solve the more standard strong formulation of the Stefan problem~\eqref{eqn: strong stefan}, provided the free boundary, $\Gamma(t)$, does not develop an interior (a so-called mushy region), cf.~\cite[Remark 2.12]{Alphonse2015Stefan}.
We refer the reader to~\cite{Crowley1979MovingBoundary,Elliott1982MovingBoundaryProblems,Gupta2003Stefan,Rodrigues1989StefanRevisted,Rubenstein1971StefanProblem} for further details on the enthalpy formulation of the Stefan problem, as well as \cite{DiPietro2015AdaptiveStefan,Elliott1987StefanError,Nochetto1985StefanLinear,Nochetto1985StefanNonlinear,Nochetto1991ParabolicFBPs,Nochetto2000AdaptivityDegenerate, Nochetto1988Degenerate} for results on finite element approximations of the two-phase Stefan problem.
To our knowledge, there is no literature concerning the numerical approximation of Stefan-type problems posed on an evolving surface.
We do however, refer the reader to recent work by Garcke and N\"urnberg \cite{Garcke2024CrystalGrowth} where they study a free boundary problem governing anisotropic crystal growth on stationary surfaces.
The corresponding system is a surface Stefan problem with surface tension and kinetic undercooling, and the authors propose a finite element approximation in the style of Barrett--Garcke--N\"urnberg type discretisations, cf.~\cite{Barrett2020ParametricFEM}.

The ESFEM developed by Dziuk and Elliott  \cite{dziuk2007finite} is a robust and efficient numerical method for the solution of PDEs posed on evolving surfaces. The method has been applied to a number of different equations and its analysis is a burgeoning area of current numerical analysis research, see e.g., \cite{Dziuk2013SurfacePDEs} for a review.  Despite their relevance in analysis and applications \cite{dibenedetto2012degenerate} the  numerical approximation of degenerate parabolic PDEs on evolving surfaces with the ESFEM or other methods  has, to our best knowledge, not been considered previously. This work is therefore, an important step in this direction and is expected to have significant impact beyond the specific case of the Stefan problem on an evolving surface which is our primary focus, cf.~Remark~\ref{remark: new l2 projection}.\\

\textbf{Main contributions:}
\begin{itemize}
    \item Our main contribution (Theorem~\ref{thm: error theorem}) is an error bound for a semi-discrete ESFEM approximation of~\eqref{eqn: stefan problem pair}.
    Our approach does not rely on mass-lumping techniques, and obtains the same order error as in the stationary, flat setting~\cite{Elliott1987StefanError,Nochetto1988Degenerate}.
    \item We introduce a new  projection-type operator (Definition~\ref{defn: new L2 projection}), which enables us to avoid the strong (and in this case infeasible) regularity assumptions often required in ESFEM error analysis~\cite{Elliott2021EvolvingFiniteElement}.
    We believe this approach will be useful in the analysis of ESFEM for other degenerate problems, cf.~Remark~\ref{remark: new l2 projection}.
    \item We introduce a new numerical method for solving the fully discrete problem by using an ``exact discretisation'' that avoids errors due to numerical quadrature under the assumption that $\beta$ is piecewise polynomial.
    This numerical method is also applicable and novel in the stationary, flat setting.
    \end{itemize}

\textbf{Outline:}
The outline of this paper is as follows.
In Section~\ref{section: stefan problem on evolving surface} we discuss the two-phase Stefan problem on an evolving surface and state our assumptions.
Section~\ref{section: evolving surfaces} recalls relevant material on evolving surfaces, and the evolving function space theory developed in~\cite{Alphonse2023EvolvingBanach,Alphonse2015Abstract}.
Section~\ref{section: ESFEM} recalls aspects of the evolving surface finite element method (ESFEM) of Dziuk and Elliott~\cite{Dziuk2013SurfacePDEs}, and states some previous results to be used in our subsequent analysis.
In Section~\ref{section: semidiscrete} we propose a spatially discrete numerical method for the approximation of \eqref{eqn: stefan problem pair} for which we prove stability (Lemma~\ref{lemma: semidiscrete stability}), continuous dependence on the data (Lemma~\ref{lemma: semidiscrete cts dependence}), and error bounds (Theorem~\ref{thm: error theorem}) under minimal regularity assumptions by introducing a new projection-type operator (Definition~\ref{defn: new L2 projection}).
To the authors' knowledge, this is the first error bound for an approximation of a free boundary problem posed on an evolving surface.
In Section~\ref{section: numerics} we propose a fully discrete numerical method and describe a novel implementation that avoids errors due to numerical quadrature of the method. This implementation has not previously been considered in the literature, even in the flat, stationary setting.
Our numerical experiments indicate this method performs similarly to quadrature-based methods.
This section is closed with numerical experiments illustrating solution behaviour that can only arise on an evolving domain, and tests that demonstrate the experimental order of convergence of the method.

\section{The Stefan problem on an evolving surface}
\label{section: stefan problem on evolving surface}
We shall assume throughout that $\Omega(t)$ is a $C^2$ evolving surface with known surface evolution, i.e.~the surface evolution is independent of the solution to~\eqref{eqn: stefan problem pair}.
By $\matdev$, we denote the material time derivative following the flow of $\mbf{V}$, $\grado$ denotes the tangential gradient on $\Omega(t)$, and $\lapo$ denotes the Laplace--Beltrami operator on $\Omega(t)$.
Note that due to the surface evolution the operators $\grado$ and $\lapo$ vary in time, but we shall ignore this in our notation.
We refer the reader to~\cite{Dziuk2013SurfacePDEs} for further details on these differential operators.
We employ notion of the form $L^p_{X}$, and $H^1_{X}$, for evolving Bochner spaces, and evolving Sobolev--Bochner spaces respectively, where $\{X(t)\}_{t \in [0,T]}$ is a family of time-dependent Banach spaces.
This is defined in greater detail in Section~\ref{section: evolving surfaces}.

\subsection{Weak formulation}
Given data $e_0 \in L^\infty(\Omega(0))$ and $f \in L^\infty_{L^\infty}$, one seeks a \emph{weak solution} pair $(e, u) \in L^{\infty}_{L^\infty} \times L^2_{H^1}$ solving \eqref{eqn: stefan problem}, subject to the inclusion \eqref{eqn: stefan inclusion}, in the sense that
\begin{align*}
    -\int_0^T \int_{\Omega(t)} e \matdev \eta + \int_0^T \int_{\Omega(t)} \grado u \cdot \grado \eta = \int_0^T \int_{\Omega(t)} f \eta + \int_{\Omega(0)} e_0 \eta(0), \quad \text{such that} \quad e \in \beta(u),
\end{align*}
for all $\eta \in L^2_{H^1} \cap H^1_{L^2}$ with $\eta(T) = 0$.
We defer discussion on the function spaces used here to Section~\ref{section: evolving surfaces}.
Here $\beta: \mbb{R} \rightrightarrows \mbb{R}$ is a set-valued map which satisfies the following assumptions.
\begin{assumption}[Admissible enthalpy functions]
	\label{assumption: enthalpy}
	We assume that $\beta: \mbb{R} \rightrightarrows \mbb{R}$ is such that:
	\begin{enumerate}
		\item[(A1)] There exists a function $\mathcal{U}: \mbb{R} \rightarrow \mbb{R}$ such that $\mathcal{U}(\beta(r)) = r$ for all $r \in \mbb{R}$.
		\item[(A2)] The inverse function, $\mathcal{U}(\cdot)$, is Lipschitz continuous and monotonically increasing (i.e.~nondecreasing).
        We shall denote the Lipschitz constant of $\mathcal{U}$ as $C_\mathcal{U}$.
		\item[(A3)] The set-valued map $\beta: \mbb{R} \rightrightarrows \mbb{R}$ is \emph{maximal monotone} in the sense that \[ (R_1 - R_2)(r_1 - r_2) \geq 0, \]
		for all $r_1, r_2 \in \mbb{R}$ and $R_i \in \beta(r_i)$ for $i=1,2$, and there exists no extension of $\beta$ which is monotone.
		\item[(A4)] The map $\beta: \mbb{R} \rightrightarrows \mbb{R}$ is \emph{strongly monotone} in the sense that there exists a constant $C_\beta > 0$ such that
		\[ (R_1 - R_2)(r_1 - r_2) \geq C_\beta (r_1 - r_2)^2, \]
		for all $r_1, r_2 \in \mbb{R}$ and $R_i \in \beta(r_i)$ for $i=1,2$.
		Equivalently this may be expressed as
		\begin{align}
			(\mathcal{U}(r_1) - \mathcal{U}(r_2))(r_1 - r_2) \geq C_\beta (\mathcal{U}(r_1) - \mathcal{U}(r_2))^2, \label{eqn: beta strong monotonicity}
		\end{align}
		for all $r_1, r_2 \in \mbb{R}$.
	\end{enumerate}
\end{assumption}

By using the inverse function $\mathcal{U}$ one finds that $u = \mathcal{U}(e)$, and hence one may equivalently write \eqref{eqn: stefan problem pair} as a single equation
\begin{align}
	\matdev e + e (\grado \cdot \mbf{V}) - \lapo \mathcal{U}(e) = f, \text{ on } \Omega(t). \label{eqn: stefan problem2}
\end{align}
In this setting one readily observes that since $\mathcal{U}'(\cdot) \geq 0$ this equation is a \emph{degenerate, quasilinear, parabolic equation} for $e$.
Using~\eqref{eqn: stefan problem2} one can now define a weak solution to~\eqref{eqn: stefan problem2} to be a function $e \in L^2_{H^{-1}} \cap L^\infty_{L^2}$ such that $\mathcal{U}(e) \in L^2_{H^1}$ and
\begin{align}
    \angled{\matdev e, \phi}_{H^{-1}(\Omega(t)) \times H^{1}(\Omega(t))} + \int_{\Omega(t)} e \phi (\grado \cdot \mbf{V}) + \int_{\Omega(t)} \grado \mathcal{U}(e) \cdot \grado \phi = \int_{\Omega(t)} f \phi, \label{eqn: weak stefan}
\end{align}
for all $\phi \in H^1(\Omega(t))$ and almost all $t \in [0,T]$ such that $e(0) = e_0$, for some $e_0 \in L^2(\Omega(0))$.
Here angled brackets denote the duality pairing between $H^{-1}(\Omega(t))$ and $H^1(\Omega(t))$.
This formulation will be the basis of our analysis.
We note that in the case of weaker data, namely $e_0 \in L^1(\Omega(0))$ and $f \in L^1_{L^1}$, one may define a weaker notion of solution~\cite[Definition 1.1]{Alphonse2015Stefan} for which our analysis is not applicable.

\begin{example}[Examples of admissible enthalpy functions]\
    As a motivating example, we consider the map
    \begin{equation}
    \beta(r) \coloneqq \begin{cases}
	\{r\}, & r < 0,\\
	[0,1], & r = 0,\\
	\{r + 1\}, & r > 0, 
    \end{cases}
    \label{eqn: enthalpy example}
    \end{equation}
    as in \cite{Alphonse2015Stefan} where the authors study the well-posedness the enthalpy formulation of the two-phase Stefan problem on an evolving surface.
    In this case it is easy to verify that Assumption~\ref{assumption: enthalpy} holds with $\mathcal{U}: \mbb{R} \rightarrow \mbb{R}$ given by
\begin{equation}\label{eq:classicalU}
\mathcal{U}(r) \coloneqq \begin{cases}
	r, & r <0,\\
	0, & r \in [0,1],\\
	r - 1, & r > 1,
\end{cases}
\end{equation}
and $C_\beta = 1$.\\

An interesting example covered by these assumptions is the map
\begin{align}
    \beta_\varepsilon(r) \coloneqq \begin{cases}
        \{ \frac{r}{\varepsilon} \}, & r \leq 0,\\
        \{\varepsilon r\}, & r > 0.
    \end{cases}
    \label{eqn: enthalpy example2}
\end{align}
This example is such that Assumption~\ref{assumption: enthalpy} holds with $\mathcal{U}_\varepsilon: \mbb{R} \rightarrow \mbb{R}$ defined by
\[
\mathcal{U}_\varepsilon(r) \coloneqq
\begin{cases}
	\varepsilon r, & r \leq 0,\\
	\frac{r}{\varepsilon}, & r > 0,
\end{cases}
\]
and $C_{\beta_{\varepsilon}}= \varepsilon$ provided $\varepsilon \leq 1$.
This is an approximation of the graph considered in the Stefan-type problems of~\cite{Alphonse2024FreeBoundaryLimit} and~\cite{Elliott2017CoupledFreeBoundary} which is obtained in the limit $\varepsilon \rightarrow 0$.
One can similarly approximate the one-phase Stefan problem by a graph of the form
\[ \beta(r) \coloneqq \begin{cases}
	\{\varepsilon r\}, & r < 0,\\
	[0,L], & r = 0,\\
	\{r + L\}, & r > 0, 
    \end{cases} \]
where $L$ denotes the latent heat.
We refer the reader to \cite{Gupta2003Stefan} for further discussion on the one-phase Stefan problem.
\end{example}
    \begin{figure}[ht]
        \centering
        \begin{subfigure}[b]{.5\linewidth}
        \centering
        \begin{tikzpicture}
                \draw[->] (-2, 0) -- (2, 0) node[right] {$r$};
                \draw[->] (0, -2) -- (0, 2) node[above] {$\beta(r)$};
                \draw[thick, red] (-2,-2)--(0,0);
                \draw[thick, red] (0,0)--(0,1);
                \draw[thick,red] (0,1)--(1,2);
                \node[left] at (0,1) {$1$};
        \end{tikzpicture}
        \caption{Plot of the graph \eqref{eqn: enthalpy example}.}
        \label{fig: enthalpy example}
        \end{subfigure}
        ~
        \begin{subfigure}[b]{.5\linewidth}
        \centering
        \begin{tikzpicture}
                \draw[->] (-2, 0) -- (2, 0) node[right] {$r$};
                \draw[->] (0, -2) -- (0, 2) node[above] {$\beta(r)$};
                \draw[thick, red] (0,0)--(2,0.2);
                \draw[thick, red] (-0.2,-2)--(0,0);
        \end{tikzpicture}
        \caption{Plot of the graph \eqref{eqn: enthalpy example2} with $\varepsilon = \frac{1}{10}$.}
        \label{fig: enthalpy example2}
        \end{subfigure}
        \caption{A plot of two example forms of $\beta$.}
        \label{fig: enthalpy examples}
    \end{figure}
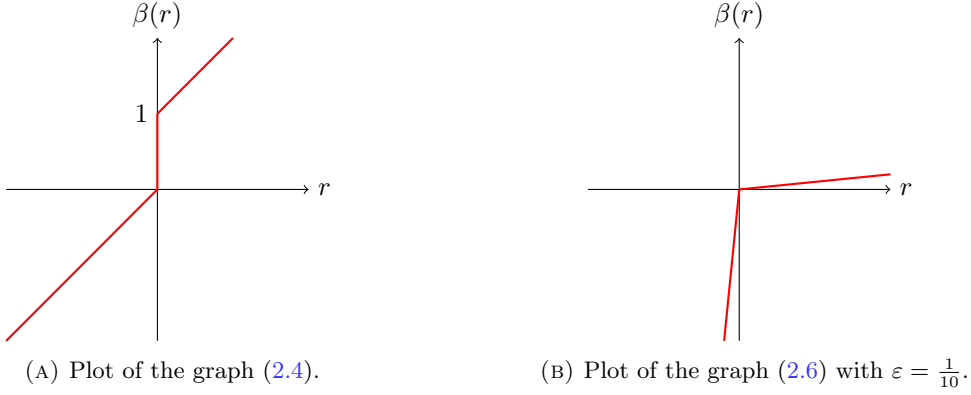

\subsection{Strong formulation of a two-phase Stefan problem}
As in \cite[Remark 2.12]{Alphonse2015Stefan} when $\beta$ is given by \eqref{eqn: enthalpy example} then a sufficiently smooth solution, $u \in H^1_{L^2} \cap L^2_{H^2}$, can be shown to solve the classical strong formulation of the Stefan problem
\begin{subequations}
\label{eqn: strong stefan}
\begin{align}
	\matdev u + (u+1)(\grado \cdot \mbf{V}) - \lapo u = f, &\qquad \text{on } \bigcup_{t \in [0,T]} \Omega_+(t) \times \{t\}, \label{eqn: strong stefan 1}\\
	\matdev u + u(\grado \cdot \mbf{V}) - \lapo u = f, &\qquad \text{on } \bigcup_{t \in [0,T]} \Omega_-(t) \times \{t\}, \label{eqn: strong stefan 2}\\
	-(\grado u_+ - \grado u_-) \cdot \boldsymbol{\mu} = V_\mu, &\qquad \text{on } \bigcup_{t \in [0,T]} \Gamma(t) \times \{t\}, \label{eqn: strong stefan 3}\\
	u = 0, &\qquad \text{on } \bigcup_{t \in [0,T]} \Gamma(t) \times \{t\}, \label{eqn: strong stefan 4}
\end{align}
\end{subequations}
pointwise almost everywhere\footnote{Here we understand almost everywhere to mean up to a set of zero $\mathcal{H}^2$ measure for~\eqref{eqn: strong stefan 1},~\eqref{eqn: strong stefan 2}, and up to a set of zero $\mathcal{H}^1$ measure for~\eqref{eqn: strong stefan 3}, \eqref{eqn: strong stefan 4}.}, where $u_\pm = u|_{\overline{\Omega_{\pm}}}$.
Here $\boldsymbol{\mu}$ denotes the unit conormal vector to $\Gamma(t)$ (i.e.~the vector which is tangential to $\Omega(t)$, normal to $\Gamma(t)$, and pointing into $\Omega_+(t)$), and $V_\mu$ denotes the conormal velocity of $\Gamma(t)$.
We illustrate the relations between $\Omega_{\pm}(t)$ and $\Gamma(t)$ appearing in this strong formulation of the Stefan problem in Figure~\ref{fig: stefan diagram}.
This holds under the assumption that there are no so-called ``mushy regions'', i.e.~regions where the interior of $\Gamma(t)$ is non-empty.
In the presence of heat sources it is known that this assumption may not hold, cf.~\cite{Bertsch1986MushyRegions,Elliott1982MovingBoundaryProblems}, and indeed we shall observe numerically that mushy regions may form on an evolving surface even in the absence of a heat source (cf. Figure~\ref{fig: stefan_mushy}).
In this setting it is clear that the enthalpy solutions we are approximating are generalised solutions of the strong formulation of Stefan problem, which does not make sense in the presence of mushy regions.
Moreover, when a solution to the strong formulation~\eqref{eqn: strong stefan} exists it is well known (see for instance~\cite{Elliott1982MovingBoundaryProblems,Gupta2003Stefan}) that this is also an enthalpy solution.

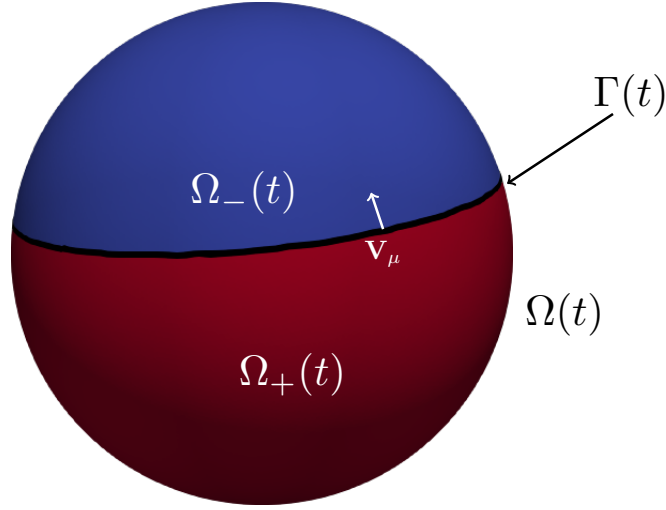
\begin{figure}[ht]
	\centering
	\scalebox{0.8}{
		\begin{tikzpicture}
			\node at (0,0) {\includegraphics[scale=0.5]{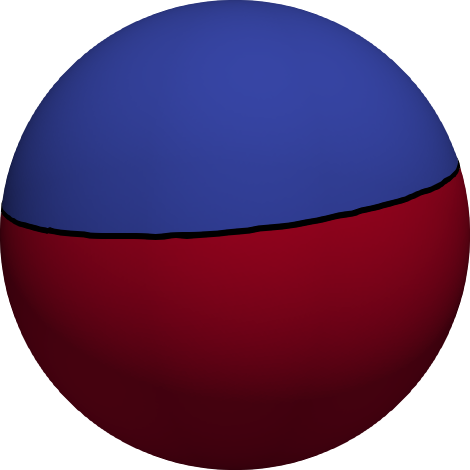}};
			\node[scale=2,white] at (0.5,-2) {${\Omega_+(t)}$};
			\node[scale=2,white] at (-0.3,1) {${\Omega_-(t)}$};
			\node[scale=2,black] at (5,-1) {${\Omega(t)}$};
			\node[scale=2] at (6.1,2.6) {${\Gamma(t)}$};
			\draw[very thick, ->]  (5.8,2.3) -- (4.05,1.15);
            \draw[white, very thick, ->] (2,0.4) -- (1.8,1);
            \node[white, scale=1.25] at (2,0) {$\mbf{V}_{\mu}$};
		\end{tikzpicture}
	}
	\caption{Diagram of the Stefan problem on an evolving surface.
    The free boundary, $\Gamma(t)$, evolves with velocity $\mbf{V}_\mu = V_\mu \boldsymbol{\mu}$ where $V_\mu$ is given by~\eqref{eqn: strong stefan 3}.}
	\label{fig: stefan diagram}
\end{figure}

\begin{remark}
\label{remark: boundary conditions}
We note that our analysis can be applied to surfaces with boundary when one considers either homogeneous Dirichlet or Neumann boundary conditions.
Similarly, our analysis is applicable to the Stefan problem in an evolving bulk domain.
The numerical scheme and our analytical results are novel even in this setting.
\end{remark}

\section{Evolving surfaces and function spaces}
\label{section: evolving surfaces}

We assume that the evolving surface, $\Omega(t)$, is $C^2$ for all $t \in [0,T]$, with a velocity field $\mbf{V} \in C^1([0,T];\mbf{C}^2(\mbb{R}^3;\mbb{R}^3))$.
It particular one finds that
\[ \sup_{t \in [0,T]} \|\mbf{V}\|_{\mbf{C}^2(\Omega(t);\mbb{R}^3)} \leq C < \infty. \]
We emphasise that this is the material velocity of the surface, $\Omega(t)$, and is known a priori and hence \emph{independent of the free boundary}.

One defines the parametrisation following the flow of $\mbf{V}$ as the unique solution, $\Phi\colon [0,T] \times \Omega(0) \rightarrow \mbb{R}^3$, of
\[ \ddt{}\Phi(t;\mbf{x}) = \mbf{V}(t; \Phi(t;\mbf{x})) \quad \forall (t,\mbf{x}) \in [0,T]\times\Omega(0), \quad \Phi(0;\mbf{x}) = \mbf{x} \quad \forall \mbf{x} \in \Omega(0). \]
By construction this is such that $\Phi(t;\Omega(0)) = \Omega(t)$, and hence one has a parametrisation of the evolving surface.
Using this parametrisation we may pushforward functions on $\Omega(0)$ by
\[ \Phi_t \psi \coloneqq \psi \circ \Phi(t)^{-1} \quad \forall \psi: \Omega(0) \rightarrow \mbb{R}, \]
and pullback functions on $\Omega(t)$ by
\[ \Phi_{-t} \chi \coloneqq \chi \circ \Phi(t) \quad \forall \chi: \Omega(t) \rightarrow \mbb{R}. \]
Using these pushforwards/pullbacks one can define \emph{evolving Bochner spaces}, which we shall denote as $L^p_X$, as in~\cite{Alphonse2023EvolvingBanach,Alphonse2015Abstract} as follows.
Given a family of Banach spaces, $\{X(t)\}_{t \in [0,T]}$, consisting of functions
\[\chi: \bigcup_{t \in [0,T]} \Omega(t) \times \{ t \} \rightarrow \mbb{R}, \]
one can define $L^p_X$, for~$p\in [1,\infty]$, as the set of functions such that $\Phi_{-t} \chi \in L^p([0,T];X(0))$, and denote the corresponding norm as $\|\cdot\|_{L^p_X}$.
We refer the reader to \cite{Alphonse2023EvolvingBanach,Alphonse2015Abstract} for further details on evolving Bocher spaces.
We say that the function space-parametrisation pairs, $\{(X(t), \Phi_t)\}_{t \in [0,T]}$, are compatible if we have also that:
\begin{enumerate}
    \item there exists a constant $C_X$ independent of $t$ such that
    \begin{align*}
        \|\Phi_t \psi\|_{X(t)} \leq C_X \|\psi\|_{X(0)} &\quad \forall \psi \in X(0),\\
        \|\Phi_{-t} \chi\|_{X(0)} \leq C_X \|\chi\|_{X(t)} &\quad \forall \chi \in X(t),
    \end{align*}
    \item for all $\chi \in X(0)$ the map $t \mapsto \|\Phi_t \chi\|_{X(t)}$ is measurable.
\end{enumerate}
We refer the reader to~\cite{Alphonse2023EvolvingBanach,Alphonse2015Abstract} for further details on evolving function spaces.
For compatible pairs, $\{(X(t), \Phi_t)\}_{t \in [0,T]}$, $L^p_X$ is a Banach space when equipped with the $\|\cdot\|_{L^p_X}$ norm defined by
\[ \|\chi\|_{L^p_X} \coloneqq \begin{cases}
    \left( \int_0^T \|\chi(t)\|_{X(t)}^p \, \dt \right)^{\frac{1}{p}}, & p \in [1,\infty),\\
    \esssup{t \in [0,T]} \|\chi(t)\|_{X(t)}, & p = \infty.
\end{cases} \]
Moreover, if $p = 2$ and $\{X(t)\}_{t \in [0,T]}$ is a family of Hilbert spaces then $L^2_X$ is a Hilbert space, where one obtains the corresponding inner product by polarisation.
Under the assumption of compatibility the function spaces $L^p_X$ inherit many of the nice properties of the usual Bochner spaces --- we refer the reader to \cite{Alphonse2023EvolvingBanach} for details.
We will be interested in the case where $X(t)$ is a Sobolev space defined over $\Omega(t)$, which we shall denote as $W^{k,q}(\Omega(t))$ for $k \in \mbb{N}$ and $q \in [1,\infty]$.
It is well known~\cite{Alphonse2023EvolvingBanach,Alphonse2015Abstract,Alphonse2015ParabolicPDEs} that our assumption on $\mbf{V}$ means that the pairs $\{(W^{k,q}(\Omega(t)),\Phi_t)\}_{t\in[0,T]}$ are compatible in the above sense for $k \in \{0,1\}$ and $q \in [1,\infty]$.
In the case $q = 2$ we shall write $H^k(\Omega(t)) \coloneqq W^{k,2}(\Omega(t))$.
We refer the reader to \cite{Aubin1982Manifolds} for further details on Sobolev spaces defined on Riemannian manifolds, and \cite[Section 2]{Alphonse2015ParabolicPDEs} for further details on Sobolev spaces on evolving hypersurfaces.

The natural notion of a time derivative on the evolving surface must take into account both the variation of the function in time, but also effects due to evolution of the surface.
As such one works with the \emph{material derivative}, defined as
\[ \matdev \chi \coloneqq \Phi_t \left( \ddt{} \Phi_{-t} \chi \right),  \]
for all sufficiently smooth functions $\chi$.
This notion can be generalised to a weak material derivative analogously to the usual time derivative, cf.~\cite[Definition 3.6]{Alphonse2023EvolvingBanach}.
We are particularly interested in the Gelfand triple setting, $H^1(\Omega(t)) \subset L^2(\Omega(t)) \equiv (L^2(\Omega(t)))^* \subset H^{-1}(\Omega(t))$, and a weak time derivative taking values in $H^{-1}(\Omega(t))$, the dual space to $H^1(\Omega(t))$.
We shall denote by $H^1_{H^{-1}}$ the evolving Sobolev--Bochner space consisting of functions
\[ H^1_{H^{-1}} \coloneqq \{ \chi \in L^2_{L^2} \mid \matdev \chi \in L^2_{H^{-1}}  \}, \]
and when the material derivative has further regularity $\matdev \chi \in L^2_{L^2}$ we shall write $\chi \in H^1_{L^2}$.

We end this section by recalling the transport theorem on an evolving surface, which we will use in our later analysis.
For this we firstly introduce the following notation.
\begin{align*}
    m(t; \phi, \psi) &\coloneqq \int_{\Omega(t)} \phi \psi,\\
    m_*(t; \mathcal{L}, \psi) &\coloneqq \angled{\mathcal{L}, \psi}_{H^{-1}(\Omega(t)) \times H^1(\Omega(t))}\\
    g(t; \phi, \psi) &\coloneqq \int_{\Omega(t)} \phi \psi (\grado \cdot \mbf{V}),\\
    a(t; \phi, \psi) &\coloneqq \int_{\Omega(t)} \grado \phi \cdot \grado \psi,\\
    b(t; \phi, \psi) &\coloneqq \int_{\Omega(t)} \left((\grado \cdot \mbf{V}) \mbb{I} - \grado \mbf{V} - (\grado \mbf{V})^T\right) \grado \phi \cdot \grado \psi,
\end{align*}
for all sufficiently smooth functions $\phi, \psi$, linear functionals $\mathcal{L} \in H^{-1}(\Omega(t))$, and where $\mbb{I}$ denotes the identity matrix.
We will omit the argument $t$ throughout, as it will be clear from context.
\begin{lemma}[{\cite[Lemma 5.2]{Dziuk2013SurfacePDEs}}]
    \label{lemma: transport theorem}
    Let $\phi, \psi \in L^2_{L^2} \cap H^1_{H^{-1}}$ then
    \[ \ddt{} m(\phi, \psi) = m_*(\matdev \phi, \psi) + m_*(\matdev \psi, \phi) + g(\phi, \psi). \]
    If we have further regularity, $\phi, \psi \in L^2_{H^1}$ and $\matdev \phi, \matdev \psi \in L^2_{H^1}$, then 
    \[ \ddt{} a(\phi, \psi) = a(\matdev \phi, \psi) + a(\matdev \psi, \phi) + b(\phi, \psi). \]
\end{lemma}
\section{The evolving surface finite element method}
\label{section: ESFEM}
\subsection{ESFEM and geometric perturbation estimates}

Let us now briefly recap some of the details of the evolving surface finite element method.
We refer the reader to~\cite{Dziuk2013SurfacePDEs,Elliott2021EvolvingFiniteElement} for further details.
Given a sufficiently smooth surface, $\Omega(0)$, and a set of vertices $\{\mbf{x}_i(0)\}_{i \in \{1,\ldots, N_h\}} \subset \Omega(0)$, we may construct a triangulated domain by appropriately connecting these vertices by edges.
Here $N_h$ denotes the number of degrees of freedom.
We denote the corresponding triangulation as $\mathcal{T}_h(0)$, which in turn defines a triangulated surface, $\Omega_h(0)$, via
\[ \Omega_h(0) \coloneqq \bigcup_{K \in \mathcal{T}_h(0)} K. \]
One then may evolve the vertices by using the parametrisation, $\Phi$, to obtain vertices $\{\mbf{x}_i(t)\}_{i \in \{1, \ldots, N_h\}}$ and a corresponding triangulation $\mathcal{T}_h(t)$.
One then defines the triangulated surface
\[\Omega_h(t) \coloneqq \bigcup_{K(t) \in \mathcal{T}_h(t)} K(t).\]
Notice that this construction yields a discrete velocity field, $\mbf{V}_h$, where is it straightforward, cf.~\cite[Section 7]{Elliott2021EvolvingFiniteElement}, to see that this is the Lagrange interpolant of $\mbf{V}$.
We shall denote the mesh size of our family of triangulated surfaces as
\[ h \coloneqq \sup_{t \in [0,T]} \max_{K(t) \in \mathcal{T}_h(t)} \diam(K(t)). \]
Throughout our analysis we will assume that our triangulations of the evolving surface are uniformly quasi-uniform in the sense of~\cite[Definition 6.29]{Elliott2021EvolvingFiniteElement}.

We now denote our (continuous Lagrange) finite element spaces as
\[ S_h(t) \coloneqq \{ \phi: \Omega_h(t) \rightarrow \mbb{R} \mid \phi|_{K(t)} \text{ is affine linear } \forall K(t) \in \mathcal{T}_h(t) \}. \]

\subsubsection{Lifts and geometric perturbation estimates}
Since our triangulated surface is not the true surface our associated finite element families are non-conforming.
As such we are committing a so-called \emph{variational crime} (cf.~\cite[Chapter 27]{Ern2021FiniteElementsII}) which we will mitigate by the use of \emph{lifts}.
For a surface, $\Omega(t)$, there exists a neighbourhood of $\Omega(t)$, denoted $\mathcal{N}(\Omega(t)) \subset \mbb{R}^3$ such that each $\mbf{x} \in \mathcal{N}(\Omega(t))$ may be uniquely expressed in Fermi coordinates
\[ \mbf{x} = \mbf{p}(t;\mbf{x}) + d(t,\mbf{x}) \boldsymbol{\nu}(t; \mbf{p}(t; \mbf{x})). \]
Here $\mbf{p}(t;\cdot)$ denotes the closest point projection onto $\Omega(t)$, $d(t;\cdot)$ denotes the signed distance function of $\Omega(t)$, and $\boldsymbol{\nu}(t;\cdot)$ denotes the outward unit normal vector on $\Omega(t)$.
We now use this to ``lift'' a function from $\Omega_h(t)$ onto $\Omega(t)$ implicitly via
\[ \eta_h^\ell(\mbf{p}(t;\mbf{x})) = \eta_h(\mbf{x}) \quad \forall \mbf{x} \in \Omega_h(t), \eta_h: \Omega_h(t) \rightarrow \mbb{R}. \]
One can similarly define an inverse lift from $\Omega(t)$ onto $\Omega_h(t)$ by
\[ \eta^{-\ell}(\mbf{x}) = \eta(\mbf{p}(t; \mbf{x})) \quad \forall \mbf{x} \in \Omega_h(t), \eta : \Omega(t) \rightarrow \mbb{R}. \]
These operations are stable in $W^{k,q}(\Omega(t))$ for $k \in \{0,1\}$ and $q \in [1,\infty]$ in the sense that there exist constants $C_{k,q} > 0$, independent of $h$, such that
\begin{align}
\frac{1}{C_{k,q}}\|\eta^{-\ell}\|_{W^{k,q}(\Omega_h(t))} \leq \|\eta\|_{W^{k,q}(\Omega(t))} \leq {C_{k,q}}\|\eta^{-\ell}\|_{W^{k,q}(\Omega_h(t))} \quad \forall \eta \in W^{k,q}(\Omega(t)). \label{eqn: lift stability}
\end{align}
We refer the reader to~\cite{Elliott2021EvolvingFiniteElement} for further details.

The evolution of the nodes of $\Omega_h(t)$ by $\mbf{V}_h$ induces a parametrisation, $\Phi^h(t): \Omega_h(0) \rightarrow \Omega^h(t)$, analogously to the definition of $\Phi(t)$.
This in turn defines a (strong) discrete material derivative by
\[ \matdev_h \chi_h \coloneqq \Phi_t^h \left( \ddt{} \Phi_{-t}^{h} \chi_h \right), \]
for sufficiently smooth functions, $\chi_h$, defined on $\Omega_h(t)$.
and one may define a weak discrete material derivative accordingly.
Using this, one may state a discrete transport theorem, analogous to Lemma~\ref{lemma: transport theorem}.
For this we introduce the following notation for bilinear forms.
\begin{align*}
    m_h(t; \phi_h, \psi_h) &\coloneqq \int_{\Omega_h(t)} \phi \psi,\\
    g_h(t; \phi_h, \psi_h) &\coloneqq \int_{\Omega_h(t)} \phi_h \psi_h (\gradoh \cdot \mbf{V}_h),\\
    a_h(t; \phi_h, \psi_h) &\coloneqq \int_{\Omega_h(t)} \gradoh \phi_h \cdot \gradoh \psi_h,\\
    b_h(t; \phi_h, \psi_h) &\coloneqq \int_{\Omega_h(t)} \left((\gradoh \cdot \mbf{V}_h) \mbb{I} - \gradoh \mbf{V}_h - (\gradoh \mbf{V}_h)^T\right) \gradoh \phi_h \cdot \gradoh \psi_h.
\end{align*}
\begin{lemma}[{\cite[Lemma 5.6]{Dziuk2013SurfacePDEs}}]
    \label{lemma: discrete transport theorem}
    Let $\phi_h, \psi_h \in \mathcal{S}_h$ then
    \begin{align*}
        \ddt{} m_h(\phi_h, \psi_h) &= m_h(\matdev_h \phi_h, \psi_h) + m_h(\matdev \psi_h, \phi_h) + g_h(\phi_h, \psi_h),\\
        \ddt{} a_h(\phi_h, \psi_h) &= a_h(\matdev_h \phi_h, \psi_h) + a_h(\matdev \psi_h, \phi_h) + b_h(\phi_h, \psi_h).
    \end{align*}
\end{lemma}
By combining the parametrisation $\Phi^h(t)$, and the lifts one may define a lifted material derivative as
\[ \matdev_\ell \chi = (\matdev_h \chi^{-\ell})^\ell, \]
for all sufficiently smooth functions, $\chi$, defined on $\Omega(t)$.
We refer the reader to~\cite{Elliott2021EvolvingFiniteElement} for further details, and the proof of the following result relating $\matdev$ and $\matdev_\ell$.

\begin{lemma}[{\cite[Lemma 9.25]{Elliott2021EvolvingFiniteElement}}]
    \label{lemma: time derivative difference}
    For a sufficiently smooth function $\eta$ one has that
    \begin{equation}
        \|\matdev \eta - \matdev_\ell \eta\|_{L^2(\Omega(t))} \leq C h^2 \|\eta\|_{H^1(\Omega(t))}. \label{eqn: time derivative difference}
    \end{equation}
\end{lemma}

One may equivalently formulate Lemma~\ref{lemma: transport theorem} using the lifted material derivative as follows.
\begin{lemma}[{\cite[Lemma 8.15]{Elliott2021EvolvingFiniteElement}}]
    \label{lemma: lifted transport theorem}
    Let $\phi, \psi \in L^2_{L^2} \cap H^1_{H^{-1}}$ then
    \[ \ddt{} m(\phi, \psi) = m_*(\matdev_\ell \phi, \psi) + m_*(\matdev_\ell \psi, \phi) + g_\ell(\phi, \psi). \]
    If we have further regularity,$\phi, \psi \in L^2_{H^1}$ and $\matdev_\ell \phi, \matdev_\ell \psi \in L^2_{H^1}$, then 
    \[ \ddt{} a(\phi, \psi) = a(\matdev_\ell \phi, \psi) + a(\matdev_\ell \psi, \phi) + b_\ell(\phi, \psi). \]
    Here we have introduced two new bilinear forms:
    \begin{align*}
        g_\ell(t; \phi, \psi) &\coloneqq \int_{\Omega(t)} \phi \psi (\grado \cdot \mbf{V}_h^\ell),\\
        b_\ell(t; \phi, \psi) &\coloneqq \int_{\Omega(t)} \left((\grado \cdot \mbf{V}_h^\ell) \mbb{I} - \grado \mbf{V}_h^\ell - (\grado \mbf{V}_h^\ell)^T\right) \grado \phi \cdot \grado \psi.
    \end{align*}
\end{lemma}

Finally, to quantify the error induced by lifting functions we have the following geometric perturbation results.
\begin{lemma}[{\cite[Lemma 9.24]{Elliott2021EvolvingFiniteElement}}]
    \label{lemma: geometric perturbations}
    For sufficiently small $h$ the following hold:
    \begin{align}
        |m(\phi, \psi) - m_h(\phi^{-\ell}, \psi^{-\ell})| &\leq Ch^2 \|\phi\|_{L^2(\Omega(t))} \|\psi\|_{L^2(\Omega(t))}, \label{eqn: geometric perturbation m}\\
        |g_\ell(\phi, \psi) - g_h(\phi^{-\ell}, \psi^{-\ell})| &\leq Ch^2 \|\phi\|_{L^2(\Omega(t))} \|\psi\|_{L^2(\Omega(t))}, \label{eqn: geometric perturbation g1}\\
        |g_\ell(\phi, \psi) - g(\phi, \psi)| &\leq Ch \|\phi\|_{L^2(\Omega(t))} \|\psi\|_{L^2(\Omega(t))}, \label{eqn: geometric perturbation g2}\\
        |a(\phi, \psi) - a_h(\phi^{-\ell}, \psi^{-\ell})| &\leq Ch^2 \|\phi\|_{H^1(\Omega(t))} \|\psi\|_{H^1(\Omega(t))}, \label{eqn: geometric perturbation a}\\
        |b_\ell(\phi, \psi) - b_h(\phi^{-\ell}, \psi^{-\ell})| &\leq Ch^2 \|\phi\|_{H^1(\Omega(t))} \|\psi\|_{H^1(\Omega(t))}, \label{eqn: geometric perturbation b1}\\
        |b_\ell(\phi, \psi) - b(\phi, \psi)| &\leq Ch \|\phi\|_{H^1(\Omega(t))} \|\psi\|_{H^1(\Omega(t))}, \label{eqn: geometric perturbation b2}
    \end{align}
    for all sufficiently smooth functions $\phi, \psi$ defined on $\Omega(t)$.
\end{lemma}

\subsection{Inverse Laplacians}

In our later analysis we will make extensive use of inverse Laplacian operators.
We now recall the definition of these operators.

\begin{definition}
    \label{defn: inverse laplace Omega}
    Let $z \in L^2(\Omega(t))$ be a function such that $\int_{\Omega(t)} z = 0$.
    Then one defines $\mathcal{G}z \in H^1(\Omega(t))$ to be the unique solution of
    \begin{align*}
    \int_{\Omega(t)} \grado \mathcal{G}z \cdot \grado \phi &= \int_{\Omega(t)}z \phi \qquad \forall \phi \in H^1(\Omega(t)),\\
    \int_{\Omega(t)} \mathcal{G}z &= 0.
    \end{align*}
    This defines a norm on the subspace of $L^2(\Omega(t))$, consisting of functions $z \in L^2(\Omega(t))$ with vanishing mean value, by
    \[ \|z\|_{-1,t} \coloneqq \sqrt{a(\mathcal{G}z, \mathcal{G}z)} = \sqrt{m(z,\mathcal{G}z)}. \]
\end{definition}

We also have the following result concerning time-differentiability.

\begin{lemma}[{\cite[Lemma 4.3]{Elliott2015CahnHilliard}}]
    \label{lemma: inverse laplacian time derivative}
    If $z \in H^1_{H^{-1}}$ then $\mathcal{G}z \in H^1_{H^1}$.
\end{lemma}

We will also require notions of a finite element inverse Laplacian defined on $\Omega_h(t)$ as in the following definitions.
\begin{definition}
    \label{defn: discrete inverse laplace}
    Let $z_h \in S_h(t)$ be a function such that $\int_{\Omega_h(t)} z_h = 0$.
    Then one defines $\GSh z_h \in S_h(t)$ to be the unique solution of
    \begin{align*}
        \int_{\Omega_h(t)} \gradoh \GSh z_h \cdot \gradoh \phi_h &= \int_{\Omega_h(t)}z_h \phi_h \qquad \forall \phi_h \in S_h(t),\\
        \int_{\Omega_h(t)} \GSh z_h &= 0.
    \end{align*}
    This defines a norm on the subspace of $S_h(t)$, consisting of functions $z_h \in S_h(t)$ with vanishing mean value, by
    \[ \|z_h\|_{S_h(t)} \coloneqq \sqrt{a_h(\GSh z_h, \GSh z_h)} = \sqrt{m_h(z_h,\GSh z_h)}. \]
\end{definition}

Next we recall that for $\Sigma(t)$ a $\mathcal{H}^2$-measurable set we shall denote its $\mathcal{H}^2$-measure by $|\Sigma(t)|$.
For such a region, and a function $z \in L^1(\Sigma(t))$, we define the mean value of $z$ over $\Sigma(t)$ as
\[ \mval{z}{\Sigma(t)} \coloneqq \frac{1}{|\Sigma(t)|}\int_{\Sigma(t)} z. \]

\begin{definition}
    \label{defn: inverse laplace Omegah}
    Let $z \in L^2(\Omega(t))$ be a function such that $\int_{\Omega(t)} z = 0$.
    Then we define $\mathcal{G}_h z \in S_h(t)$ to be the unique solution of
    \begin{align*}
    \int_{\Omega_h(t)} \gradoh \mathcal{G}_h z \cdot \gradoh \phi_h &= \int_{\Omega_h(t)} \left( z^{-\ell} - \mval{z^{-\ell}}{\Omega_h(t)}\right) \phi_h \qquad \forall \phi_h \in S_h(t),\\
    \int_{\Omega_h(t)} \Gh z &= 0.
    \end{align*}
    This defines a norm on the subspace of $L^2(\Omega(t))$, consisting of functions $z \in L^2(\Omega(t))$ with vanishing mean value, by
    \[ \|z\|_{-h,t} \coloneqq \sqrt{a_h(\mathcal{G}_hz, \mathcal{G}_hz)} = \sqrt{m_h\left(z^{-\ell} - \mval{z^{-\ell}}{\Omega_h(t)}, \mathcal{G}_h z\right)}. \]
\end{definition}

We can relate $\G$ and $\Gh$ through the following error estimate, in analogy to the usual inverse Laplacians on Euclidean domains.

\begin{lemma}
    \label{lemma: inverse laplacian error}
    Let $z \in L^2(\Omega(t))$ be such that $\int_{\Omega(t))} z = 0$.
    Then, for sufficiently small $h$, one has
    \begin{align}
        \|\mathcal{G}z - \mathcal{G}_h^\ell z\|_{L^2(\Omega(t))} + h \|\grado(\mathcal{G}z - \mathcal{G}_h^\ell z)\|_{L^2(\Omega(t))} \leq Ch^2 \|z\|_{L^2(\Omega(t))}, \label{eqn: inverse laplacian error}
    \end{align}
    for a constant, $C$, independent of $t$, $z$ and $h$.
    Here we are using notation $\mathcal{G}_h^\ell z \coloneqq (\mathcal{G}_h z)^\ell$.
\end{lemma}
\begin{proof}
    This is essentially the standard error bound for (piecewise linear) surface finite element approximations of the Laplace equation, but we spell out some of the details nonetheless.
    One appeals to~\cite[Theorem 4.9]{Dziuk2013SurfacePDEs} to see that
    \begin{align*}
        \|\mathcal{G}z - \mathcal{G}_h^\ell z\|_{L^2(\Omega(t))} &\leq Ch^2\|z\|_{L^2(\Omega(t))} + C\left\| \mval{z^{-\ell}}{\Omega_h(t))} \right\|_{L^2(\Omega(t))},\\
        \|\grado(\mathcal{G}z - \mathcal{G}_h^\ell z))\|_{L^2(\Omega(t))} &\leq Ch\|z\|_{L^2(\Omega(t))} + C\left\| \mval{z^{-\ell}}{\Omega_h(t))} \right\|_{L^2(\Omega(t))},
    \end{align*}
    and so we need only bound this mean value term.
    For this we use Lemma~\ref{lemma: geometric perturbations} to see that
    \[ \left| \mval{z^{-\ell}}{\Omega_h(t)} \right| = \frac{1}{|\Gamma_h(t)|}\left| m_h(z^{-\ell},1) - m(z,1) \right| \leq Ch^2 \|z\|_{L^2(\Omega(t))}, \]
    where we have also used the assumption that $\int_{\Omega(t)} z = 0$.
    \eqref{eqn: inverse laplacian error} now follows.
\end{proof}

One can similarly relate $\mathcal{G}$ and $\mathcal{G}_h$ to $\GSh$, but this will not be useful in our subsequent analysis.

\begin{remark}
    \label{remark: discrete inverse laplacian}
    Arguably a more natural definition for $\mathcal{G}_h z \in S_h(t)$ would be the unique solution of
    \begin{align*}
        \int_{\Omega_h(t)} \gradoh \mathcal{G}_h z \cdot \gradoh \phi_h &= \int_{\Omega(t)} z \phi_h^\ell \qquad \forall \phi_h \in S_h(t),\\
        \int_{\Omega_h(t)} \mathcal{G}_h z &= 0.
    \end{align*}
    However, in this case it appears that this notion of inverse Laplacian would require higher regularity and yield a lower order error estimate --- hence we instead use Definition~\ref{defn: inverse laplace Omegah}.
    To see this, one observes that
    \[ \int_{\Omega(t)} z \phi_h^\ell = \int_{\Omega_h(t)} \Lambda_h z  \phi_h, \]
    where $\Lambda_h: L^2(\Omega(t)) \rightarrow S_h(t)$ is the $L^2$ projection onto $S_h(t)$.
    In this case~\cite[Theorem 4.9]{Dziuk2013SurfacePDEs} now yields
    \begin{align*}
        \|\mathcal{G}z - \mathcal{G}_h^\ell z\|_{L^2(\Omega(t))} &\leq Ch^2\|z\|_{L^2(\Omega(t))} + C\|z - (\Lambda_h z)^\ell\|_{L^2(\Omega(t))},\\
        \|\grado(\mathcal{G}z - \mathcal{G}_h^\ell z)\|_{L^2(\Omega(t))} &\leq Ch\|z\|_{L^2(\Omega(t))} + C\|z - (\Lambda_h z)^\ell\|_{L^2(\Omega(t))},
    \end{align*}
    and since our mesh is uniformly quasi-uniform one finds that, cf.~\cite{Thomee2006GalerkinFEM},
    \[ \|z - (\Lambda_h z)^\ell\|_{L^2(\Omega(t))} \leq C h \|z\|_{H^1(\Omega(t))}, \]
    provided that $z \in H^1(\Omega(t))$.
\end{remark}

\section{A semi-discrete evolving surface finite element method}
\label{section: semidiscrete}
We now introduce our spatially-discrete finite element method to be analysed.
Given initial data, $e_h^0 \in S_h(0)$, one finds $e_h \in \mathcal{S}_h$, where
\[ \mathcal{S}_h \coloneqq \{ \chi_h \mid \Phi_{-t}^h \chi_h \in C^1([0,T];S_h(0)) \}, \]
such that
\begin{align}
	m_h(\matdev_h e_h, \phi_h) + g_h(e_h, \phi_h) + a_h(\mathcal{U}(e_h), \phi_h) = m_h(f_h, \phi_h) \quad \forall \phi_h \in S_h(t), \label{eqn: semidiscrete stefan}
\end{align}
for almost all $t \in [0,T]$, and such that $e_h(0) = e_h^0$.

\begin{remark}
    \label{remark: mass-lumping}
    In existing literature~\cite{Elliott1987StefanError,Elliott1982MovingBoundaryProblems,Nochetto1991ParabolicFBPs} on the Stefan problem on a flat domain one often assumes mesh acuteness and uses mass-lumped finite elements to allow for a discrete maximum principle.
    We refer the reader to~\cite{Barrenechea2024DiscreteMaximum} for an overview on discretisations allowing discrete maximum principles.
    On an evolving surface this is problematic, as it is known that the nodal evolution may cause an initially acute mesh to lose this property, cf.~\cite{Deckelnick2019HamiltonJacobi} and~\cite[Remark 2.16]{Elliott2026Logarithmic}.
    As such our analysis will avoid the use of discrete maximum principles.
    One may wish to explore alternate numerical methods for this problem --- we defer further discussion in this regard to Section~\ref{section: numerics}.
\end{remark}

\subsection{Stability}

\begin{lemma}
    \label{lemma: semidiscrete stability}
    Let $\beta$ satisfy Assumption~\ref{assumption: enthalpy}.
    There exists a function $e_h \in \mathcal{S}_h$ solving~\eqref{eqn: semidiscrete stefan} for all $\phi_h \in S_h(t)$ for almost all $t \in [0,T]$ and such that $e_h(0) = e_h^0$.
    Moreover this function is such that
    \begin{gather}
    \begin{multlined}
        \sup_{t \in [0,T]} \|e_h\|_{L^2(\Omega_h(t))}^2 + \frac{1}{C_{\mathcal{U}}} \int_0^T \|\gradoh \mathcal{U}(e_h)\|_{L^2(\Omega_h(t))}^2\\
        \leq C\left(\|e_h^0\|_{L^2(\Omega_h(0))}^2 + \int_0^T \|f_h\|_{L^2(\Omega_h(t))}^2 \right), \label{eqn: semidiscrete energy estimate1}
    \end{multlined}\\
        \int_{0}^T \|\matdev_h e_h\|_{H^{-1}(\Omega_h(t))}^2 \leq C\left(\|e_h^0\|_{L^2(\Omega_h(0))}^2 + \int_0^T \|f_h\|_{L^2(\Omega_h(t))}^2 \right), \label{eqn: semidiscrete energy estimate2}
    \end{gather}
    for constants $C$ independent of $e_h$ and $h$, but depending on $T$. 
\end{lemma}
\begin{proof}
    The local-in-time existence of such a function is a straightforward result of standard ODE theory, since the bilinear forms are differentiable in $t$, and the nonlinearities are Lipschitz continuous.
    As is standard, we now show this solution exists on the interval $[0,T]$ by establishing energy estimates.
    For this, we test~\eqref{eqn: semidiscrete stefan} with $e_h$ to find that
    \begin{equation}
        m_h(\matdev_h e_h, e_h) + g_h(e_h, e_h) + a_h(\mathcal{U}(e_h), e_h) = m_h(f_h, e_h). \label{eqn: semidisc energy pf1}
    \end{equation}
    By using Lemma~\ref{lemma: discrete transport theorem} we find that
    \[ m_h(\matdev_h e_h, e_h) + g_h(e_h, e_h) = \frac{1}{2}\ddt{} \|e_h\|_{L^2(\Omega_h(t))}^2 + \frac{1}{2} g_h(e_h, e_h). \]
    We now claim that the $a_h(\cdot, \cdot)$ term is bounded below by
    \[ a_h(\mathcal{U}(e_h), e_h) \geq \frac{1}{C_\mathcal{U}}a_h(\mathcal{U}(e_h), \mathcal{U}(e_h)). \]
    To see this, observe that by H\"older's inequality, and the Lipschitz continuity of $\mathcal{U}(\cdot)$, one has
    \[ a_h(\mathcal{U}(e_h), \mathcal{U}(e_h)) = \int_{\Omega_h(t)} \mathcal{U}'(e_h) \gradoh e_h \cdot \gradoh \mathcal{U}(e_h) \leq C_\mathcal{U} \int_{\Omega_h(t)} \left| \gradoh e_h \cdot \gradoh \mathcal{U}(e_h) \right|, \]
    and this rightmost integral may be written as
    \[\int_{\Omega_h(t)} \left| \gradoh e_h \cdot \gradoh \mathcal{U}(e_h) \right| = \int_{\Omega_h(t)} \gradoh e_h \cdot \gradoh \mathcal{U}(e_h), \]
    since $\gradoh \mathcal{U}(e_h) = \mathcal{U}'(e_h) \gradoh e_h $ and $\mathcal{U}'(\cdot)$ is nonnegative.
    We now combine these two facts in \eqref{eqn: semidisc energy pf1} to obtain
    \begin{equation}
        \frac{1}{2}\ddt{} \|e_h\|_{L^2(\Omega_h(t))}^2 + \frac{1}{C_\mathcal{U}} \|\gradoh \mathcal{U}(e_h)\|_{L^2(\Omega_h(t))}^2 \leq m_h(f_h, e_h) -\frac{1}{2} g_h(e_h, e_h). \label{eqn: semidisc energy pf2}
    \end{equation}
    By using Young's inequality and the smoothness of $\mbf{V}$ one finds that
    \begin{equation}
        \frac{1}{2}\ddt{} \|e_h\|_{L^2(\Omega_h(t))}^2 + \frac{1}{C_\mathcal{U}} \|\gradoh \mathcal{U}(e_h)\|_{L^2(\Omega_h(t))}^2 \leq \|f_h\|_{L^2(\Omega_h(t))}^2 + C\|e_h\|_{L^2(\Omega_h(t))}^2, \label{eqn: semidisc energy pf3}
    \end{equation}
    for a constant, $C$, depending on $\mbf{V}$.
    \eqref{eqn: semidiscrete energy estimate1} now follows by integrating~\eqref{eqn: semidisc energy pf3} in time and using a Gr\"onwall inequality.
    One now obtains~\eqref{eqn: semidiscrete energy estimate2} as a consequence of ~\eqref{eqn: semidiscrete energy estimate1} since
    \[ \frac{m_h(\matdev_h e_h, \phi_h)}{\|\phi_h\|_{H^1(\Omega_h(t))}} \leq \|\gradoh \cdot \mbf{V}_h\|_{L^\infty(\Omega_h(t))}\|e_h\|_{L^2(\Omega_h(t))} + \|\gradoh \mathcal{U}(e_h)\|_{L^2(\Omega_h(t)))} + \|f_h\|_{L^2(\Omega_h(t))}. \]
\end{proof}
Next we will show a result concerning continuous dependence on the data.
For this we require the following Ritz projection and error bound.
\begin{definition}
    \label{defn: ritz projection}
    Given $z_h \in H^1(\Omega_h(t))$ we define the Ritz projection, $R_h z_h \in S_h(t)$, to be the unique function such that
    \begin{align*}
        a_h(R_h z_h, \phi_h) &= a_h(z_h, \phi_h) \quad \forall \phi_h \in S_h(t),\\
        \mval{R_h z_h}{\Omega_h(t)} &= \mval{z_h}{\Omega_h(t)}.
    \end{align*}
\end{definition}
\begin{lemma}
    \label{lemma: ritz error}
    For $z_h \in H^1(\Omega_h(t))$ and $R_h z_h \in S_h(t)$ as defined above one has
    \begin{align}
        \|R_h z_h\|_{H^1(\Omega_h(t))} &\leq C \|z_h\|_{H^1(\Omega_h(t))}, \label{eqn: ritz stability}\\
        \|z_h - R_h z_h\|_{L^2(\Omega_h(t))} &\leq Ch\|z_h\|_{H^1(\Omega_h(t))}, \label{eqn: ritz error}
    \end{align}
    where $C$ denotes a constant independent of $h$ and $t \in [0,T]$.
\end{lemma}
\begin{proof}
    This proof is follows by same arguments used in the usual error analysis for the Ritz projection, cf.~\cite{Elliott2021EvolvingFiniteElement,Thomee2006GalerkinFEM}.
\end{proof}
\begin{lemma}
	\label{lemma: semidiscrete cts dependence}
	Let $e_{h,1}^0, e_{h,2}^0 \in S_h(0)$ such that $\int_{\Omega_h(0)} e_{h,1}^0 = \int_{\Omega_h(0)} e_{h,2}^0$, and $f_{h,1}, f_{h,2} \in \mathcal{S}_h$ such that $\int_{\Omega_h(t)} f_{h,1} = \int_{\Omega_h(t)} f_{h,2}$ for all $t \in [0,T]$.
	Then letting $e_{h,i}$, for $i = 1,2$, denote a solution of
	\[ m_h(\matdev_h e_{h,i}, \phi_h) + g_h(e_{h,i}, \phi_h) + a_h(\mathcal{U}(e_{h,i}), \phi_h) = m_h(f_{h,i}, \phi_h) \]
	for all $\phi_h \in S_h(t)$, for all $t \in [0,T]$, and such that $e_{h,i}(0) = e_{h,i}^0$, one has
	\begin{multline}
		\|e_{h,1} - e_{h,2}\|_{S_h(T)}^2 + C_\beta \int_{0}^T \|\mathcal{U}(e_{h,1}) - \mathcal{U}(e_{h,2})\|_{L^2(\Omega_h(t))}^2\\
		\leq C(T) \left(h + \|e_{h,1}^0 - e_{h,2}^0\|_{S_h(0)}^2 + \int_0^T \|f_{h,1} - f_{h,2}\|_{S_h(t)}^2 \right), \label{eqn: semidiscrete cts dependence}
	\end{multline}
	for a constant $C$ independent of $h$, but depending on $T$, $e_{h,1}^0$, $e_{h,2}^0$, and $C_\mathcal{U}$.
\end{lemma}
\begin{proof}
	 If we define a function
	 \[ E_h \coloneqq e_{h,1} - e_{h,2}, \]
	 then immediately one finds that
	 \begin{align}
	 	m_h(\matdev_h E_h, \phi_h) + g_h(E_{h}, \phi_h) + a_h(\mathcal{U}(e_{h,1}) - \mathcal{U}(e_{h,2}), \phi_h) = m_h(f_{h,1} - f_{h,2}, \phi_h), \label{eqn: semidiscrete stability pf1}
	 \end{align}
	 for all $\phi_h \in \mathcal{S}_h$ and $E_h(0) = e_{h,1}^0 - e_{h,2}^0$.
	 Testing \eqref{eqn: semidiscrete stability pf1} with $\phi_h \equiv 1$ we see that, by the assumptions made on $f_{h,1}, f_{h,2}$,
	 \[ 0 = m_h(\matdev_h E_h, 1) + g_h(E_{h}, 1) = \ddt{} m_h(E_h, 1). \]
	 Our assumptions on $e_{h,1}^0$ and $e_{h,2}^0$ now imply that
	 \[ \int_{\Omega_h(t)} E_h = 0, \quad \text{for a.e. } t \in [0,T].\]
	 Hence we find that $\GSh E_h$ is well-defined.
	 Testing \eqref{eqn: semidiscrete stability pf1} with $\GSh E_h$ we find that
	 \begin{equation}
	 	m_h(\matdev_h E_h, \GSh E_h) + g_h(E_{h}, \GSh E_h) + a_h( \mathcal{U}(e_{h,1}) - \mathcal{U}(e_{h,2}), \GSh E_h) = m_h(f_{h,1} - f_{h,2}, \GSh E_h). \label{eqn: semidiscrete stability pf2}
	 \end{equation}
	 From Definition~\ref{defn: discrete inverse laplace}, Definition~\ref{defn: ritz projection}, and~\eqref{eqn: beta strong monotonicity} we find that
	 \begin{align*}
	     a_h(\mathcal{U}(e_{h,1}) - \mathcal{U}(e_{h,2}), \GSh E_h) &= a_h(R_h\left(\mathcal{U}(e_{h,1}) - \mathcal{U}(e_{h,2})\right), \GSh E_h)\\
         &= m_h(R_h\left(\mathcal{U}(e_{h,1}) - \mathcal{U}(e_{h,2})\right), E_h)\\
         &\geq C_{\beta} \|\mathcal{U}(e_{h,1}) - \mathcal{U}(e_{h,2})\|_{L^2(\Omega_h(t))}^2\\
         &+ m_h(R_h\left(\mathcal{U}(e_{h,1}) - \mathcal{U}(e_{h,2})\right) - \mathcal{U}(e_{h,1}) - \mathcal{U}(e_{h,2}), E_h). 
	 \end{align*}
	 Next we observe that from Lemma~\ref{lemma: discrete transport theorem} and the definition of $\|\cdot\|_{S_h(t)}$ we have
	 \begin{align*}
	 	m_h(\matdev_h E_h, \GSh E_h) + g_h(E_{h}, \GSh E_h) &= \ddt{} \|E_h\|_{S_h(t)}^2 - m_h(E_h, \matdev_h \GSh E_h)\\
	 	&= \ddt{} \|E_h\|_{S_h(t)}^2 - a_h(\GSh E_h, \matdev_h \GSh E_h)\\
	 	&= \frac{1}{2} \ddt{} \|E_h\|_{S_h(t)}^2 + \frac{1}{2} b_h(\GSh E_h, \GSh E_h),
	\end{align*}
	where the final equality follows from Lemma~\ref{lemma: discrete transport theorem} since
	\[  a_h(\GSh E_h, \matdev_h \GSh E_h) = \frac{1}{2} \ddt{} \|E_h\|_{S_h(t)}^2 - \frac{1}{2} b_h(\GSh E_h, \GSh E_h). \]
	Combining these facts together in \eqref{eqn: semidiscrete stability pf2} we see that
	\begin{align*}
    \frac{1}{2} \ddt{} \|E_h\|_{S_h(t)}^2 + C_{\beta} \|\mathcal{U}(e_{h,1}) - \mathcal{U}(e_{h,2})\|_{L^2(\Omega_h(t))}^2 &\leq m_h(f_{h,1} - f_{h,2}, \GSh E_h) + \frac{1}{2} b_h(\GSh E_h, \GSh E_h)\\
    &+ m_h(R_h\left(\mathcal{U}(e_{h,1}) - \mathcal{U}(e_{h,2})\right) - \mathcal{U}(e_{h,1}) - \mathcal{U}(e_{h,2}), E_h).
    \end{align*}
	From our smoothness assumptions on $\mbf{V}$ there exists a constant, $C$, independent of $h$ such that
	\[ |b_h(\GSh E_h, \GSh E_h)| \leq C\|E_h\|_{-h,t}^2. \]
    Next we use Lemma~\ref{lemma: ritz error} to see that
    \[ |m_h(R_h\left(\mathcal{U}(e_{h,1}) - \mathcal{U}(e_{h,2})\right) - \mathcal{U}(e_{h,1}) - \mathcal{U}(e_{h,2}), E_h)| \leq Ch\|\mathcal{U}(e_1) - \mathcal{U}(e_2)\|_{H^1(\Omega_h(t))}\|E_h\|_{L^2(\Omega_h(t)}, \]
    for a constant $C$ independent of $h$.
    We use the Lipschitz continuity of $\mathcal{U}(\cdot)$ to see that
    \begin{align*}\|\mathcal{U}(e_1) - \mathcal{U}(e_2)\|_{H^1(\Omega_h(t))}\|E_h\|_{L^2(\Omega_h(t)} &\leq C_{\mathcal{U}}\|E_h\|_{L^2(\Omega_h(t))}^2\\
    &+ \|\gradoh(\mathcal{U}(e_1) - \mathcal{U}(e_2))\|_{L^2(\Omega_h(t))}\|E_h\|_{L^2(\Omega_h(t))},
    \end{align*}
    where we recall that $C_\mathcal{U}$ denotes the Lipschitz constant of $\mathcal{U}$.
    From Lemma~\ref{lemma: semidiscrete stability} and H\"older's inequality one now obtains
    \[ \int_0^T \|\mathcal{U}(e_1) - \mathcal{U}(e_2)\|_{H^1(\Omega_h(t))}\|E_h\|_{L^2(\Omega_h(t)} \leq C\left( \|e_{h,1}^0\|_{L^2(\Omega_h(0))}^2 + \|e_{h,2}^0\|_{L^2(\Omega_h(0))}^2 \right), \]
    where the constant $C$ depends on $T$ and $C_\mathcal{U}$.
	Finally, we use Definition~\ref{defn: discrete inverse laplace}, after noting that $\GSh(f_{h,1} - f_{h,2})$ is well-defined, along with Young's inequality and Lemma~\ref{lemma: semidiscrete stability} to find that
	\begin{align*}\|E_h\|_{S_h(t)}^2 + 2C_{\beta} \int_0^T \|\mathcal{U}(e_{h,1}) -  \mathcal{U}(e_{h,2})\|_{L^2(\Omega_h(t))}^2 &\leq Ch + \|E_h\|_{S_h(0)}^2\\
    &+ \int_0^T \|f_{h,1} - f_{h,2}\|_{S_h(t)}^2 + C\int_0^T \|E_h\|_{S_h(t)}^2 .
    \end{align*}
	\eqref{eqn: semidiscrete cts dependence} now follows by an application of the Gr\"onwall inequality.
\end{proof}
\begin{remark}\
    \label{remark: mean value cts dependence}
	\begin{enumerate}
	\item One can drop the requirements that $\int_{\Omega_h(0)} e_{h,1}^0 = \int_{\Omega_h(0)} e_{h,2}^0$ and $\int_{\Omega_h(t)} f_{h,1} = \int_{\Omega_h(t)} f_{h,2}$ for all $t \in [0,T]$ by separately treating the mean values.
    This is a straightforward modification of the above lemma, but is notationally more confusing and so we do note treat this here.
	\item This continuous dependence result is a weaker version of the result shown in \cite{Alphonse2015Stefan} wherein the authors prove continuous dependence in the $L^1$ norm (see \cite[Theorem 1.4/1.5]{Alphonse2015Stefan}) for the non-discretised problem.
\end{enumerate}
\end{remark}

\subsubsection{A method with quadrature}
We now end this section with a result comparing our numerical method to a numerical method using a quadrature rule on the nonlinear term.
Notice that we do note consider the effects of mass-lumping since our results will not use discrete maximum principle arguments.
This is in contrast to previous numerical analysis of the Stefan problem~\cite{Elliott1987StefanError}.
For this we firstly recall the following result.
\begin{lemma}[{\cite[Lemma 2.27]{Elliott2026Logarithmic}}]
    \label{lemma: quadrature error}
    Let $\mathcal{U} \in C^{0,1}(\mbb{R})$ be monotonically increasing (i.e.~nondecreasing).
    Then for $I_h$ denoting the Lagrange interpolant on $\Omega_h$ one has that
    \begin{align}
        \label{eqn: quadrature error}
        \|I_h \mathcal{U}(\phi_h) - \mathcal{U}(\phi_h) \|_{L^2(\Omega_h(t))} \leq C h \|\gradoh I_h \mathcal{U}(\phi_h)\|_{L^2(\Omega_h(t))} \quad \forall \phi_h \in S_h(t),
    \end{align}
    for a constant $C$ independent of $h$ and $t$.
\end{lemma}

\begin{lemma}
    \label{lemma: numerical integration comparison}
    Let $\beta$ satisfy Assumption~\ref{assumption: enthalpy}.
    Let $e_{h,1}$ denote the solution to \eqref{eqn: semidiscrete stefan}, and let $e_{h,2}$ solve
    \[ m_h(\matdev_h e_{h,2}, \phi_h) + g_h(e_{h,2}, \phi_h) + a_h(I_h \mathcal{U}(e_{h,2}), \phi_h) = m_h(f_h, \phi_h) \quad \forall \phi_h \in S_h(t), \]
    for almost all $t \in [0,T]$ and such that $e_{h,2}(0) = e_h^0$.
    Then if $e_{h,2}$ is such that
    \begin{align}
        \sup_{t \in [0,T]} \|e_{h,2}\|_{L^2(\Omega_h(t))} + \int_0^T \|\gradoh I_h \mathcal{U}(e_{h,2})\|_{L^2(\Omega_h(t))}  \leq C \|e_h^0\|_{L^2(\Omega_h(0))},
    \end{align}
    for a constant $C$ independent of $h$.
    Then 
    \begin{align*}\|E_h\|_{S_h(T)}^2 + 2C_{\beta} \int_0^T \|\mathcal{U}(e_{h,1}) -  \mathcal{U}(e_{h,2})\|_{L^2(\Omega_h(t))}^2 &\leq C(T)h \|e_h^0\|_{L^2(\Omega_h(0))}^2,
    \end{align*}
    for a constant $C(T)$ a constant independent of $h$ but depending on $T$.
\end{lemma}
\begin{proof}
    As in the previous lemma, we define $E_h \coloneqq e_{h,1} - e_{h,2}$.
    It is now straightforward to verify that $E_h$ is such that
    \begin{align}
        m_h(\matdev_h E_h, \phi_h) + g_h(E_h, \phi_h) + a_h(\mathcal{U}(e_{h,1})-I_h \mathcal{U}(e_{h,2}), \phi_h) = 0 \quad \forall \phi_h \in S_h(t), \label{eqn: quadrature pf1}
    \end{align}
    for almost all $t \in [0,T]$ and such that $E_h(0) = 0$.
    The proof from here is now similar to that of Lemma~\ref{lemma: semidiscrete cts dependence}.
    We test~\eqref{eqn: quadrature pf1} with $\GSh E_h$, and use Definition~\ref{defn: discrete inverse laplace} and Definition~\ref{defn: ritz projection} to see that
    \[ m_h(\matdev_h E_h, \GSh E_h) + g_h(E_h, \GSh E_h) + m_h(R_h(\mathcal{U}(e_{h,1})-I_h \mathcal{U}(e_{h,2})), E_h) = 0. \]
    We now recall from the proof of Lemma~\ref{lemma: semidiscrete cts dependence} that
    \begin{align*}
	 	m_h(\matdev_h E_h, \GSh E_h) + g_h(E_{h}, \GSh E_h) = \frac{1}{2} \ddt{} \|E_h\|_{S_h(t)}^2 + \frac{1}{2} b_h(\GSh E_h, \GSh E_h),
	\end{align*}
    from which one can now observe that
    \begin{align}
        \frac{1}{2} \ddt{} \|E_h\|_{S_h(t)}^2 + m_h(\mathcal{U}(e_{h,1}) - \mathcal{U}(e_{h,2}), E_h) &= m_h(\mathcal{U}(e_{h,1}) - R_h \mathcal{U}(e_{h,1}),E_h)\notag \\
        &+ m_h(I_h \mathcal{U}(e_{h,2}) - \mathcal{U}(e_{h,2}), E_h) \label{eqn: quadrature pf2}\\
        &-\frac{1}{2} b_h(\GSh E_h, \GSh E_h), \notag
    \end{align}
    where we have also used the linearity of the Ritz projection, and the fact that $R_h I_h \mathcal{U}(e_{h,2}) = I_h \mathcal{U}(e_{h,2})$.
    The proof will now follow by the same ideas as in the proof of Lemma~\ref{lemma: semidiscrete cts dependence} where the only new ingredient we require is the use of Lemma~\ref{lemma: quadrature error} to see that
    \[ |m_h(I_h \mathcal{U}(e_{h,2}) - \mathcal{U}(e_{h,2}), E_h)| \leq Ch \|\gradoh I_h \mathcal{U}(e_{h,2})\|_{L^2(\Omega_h(t))} \left( \|e_{h,1}\|_{L^2(\Omega_h(t))} + \|e_{h,2}\|_{L^2(\Omega_h(t))} \right). \]
    We omit further details.
\end{proof}
\begin{remark}
    One can verify the hypotheses of this lemma with minor adaptions to the proof of Lemma~\ref{lemma: semidiscrete stability} to verify the hypotheses.
    For this one wishes to use~\cite[Lemma 2.28]{Elliott2026Logarithmic} which requires some notion of mesh-acuteness, which is problematic on evolving surfaces as discussed in Remark~\ref{remark: mass-lumping}.
\end{remark}

\subsection{Error bounds}

We now prove our main result, Theorem~\ref{thm: error theorem}, for which we shall make the following assumption on the data.

\begin{assumption}[Data approximation]
    \label{assumption: data}
    We assume that the initial data, $e_h^0 \in S_h(0)$, is such that
    \[ \|e_h^0\|_{L^2(\Omega_h(0))} \leq C \|e_0\|_{L^2(\Omega(0))} \text{ and } \int_{\Omega_h(0)} e_h^0 = \int_{\Omega(0)} e_0, \]
    where $C$ is a constant independent of $h$.
    Likewise we assume that $f_h(t) \in S_h(t)$ is such that
    \[ \|f_h(t)\|_{L^2(\Omega_h(t))} \leq C \|f(t)\|_{L^2(\Omega(t))} \text{ and } \int_{\Omega_h(t)} f_h(t) = \int_{\Omega(t)} f(t) \quad \text{for a.e. } t \in [0,T], \]
    where $C$ is a constant independent of $h$ and $t$.
\end{assumption}
Examples of data satisfying this are $e_h^0 = \Lambda_h(0) e_0$ and $f_h(t) = \Lambda_h(t) f(t)$, where $\Lambda_h(t) : L^2(\Omega(t)) \rightarrow S_h(t)$ denotes the $L^2$ projection onto $S_h(t)$, or suitably defined\footnote{For example the Ritz projection of \cite[Definition 3.6]{Elliott2021EvolvingFiniteElement}, rather than that in our Definition~\ref{defn: ritz projection}.} Ritz projections.

\begin{remark}
    \label{remark: interpolant initial data}
    For practical methods one may wish to use the Lagrange interpolant of the data, for example $e_h^0 = I_h e_0^{-\ell}$ for $e_0 \in C^0(\Omega(0)) \cap W^{1,1}(\Omega(0))$.
    In this case it is sufficient for our analysis to note that
    \[ \|I_h e_0^{-\ell}\|_{C^0(\Omega_h(0))} \leq C \|e_0\|_{C^0(\Omega(0))}, \]
    but the mean value condition may not hold. 
    Notice however that standard error bounds for the Lagrange interpolant (cf.~\cite[Corollary 7.1]{Elliott2021EvolvingFiniteElement}) imply that
    \[ \left|\int_{\Omega_h(0))} I_h e_0^{-\ell} - \int_{\Omega(0)} e_0\right| \leq C h \|e_0\|_{W^{1,1}(\Omega(0))}, \]
    and so a suitable choice of initial data would be
    \[e_h^0 = I_h e_0^{-\ell} + \frac{1}{|\Omega_h(0)|}\left(\int_{\Omega(0)} e_0 - \int_{\Omega_h(0))} I_h e_0^{-\ell}\right).\]
    Moreover, owing to Lemma~\ref{lemma: semidiscrete cts dependence} and Remark~\ref{remark: mean value cts dependence} one may in practice neglect these mean value contributions.
    In practice the initial data may be discontinuous in $\Omega(0)$ but piecewise (H\"older) continuous in the phases $\Omega_{\pm}(0)$ --- in this case we may nonetheless choose the initial data as the Lagrange interpolant (see \cite[Remark 6]{Nochetto1988Degenerate}).
\end{remark}

\subsubsection{Preliminaries}

We now introduce some preliminary results to be used in our subsequent error analysis. 
Firstly we define an $L^2$ projection-type operator from $\Omega_h(t)$ onto $\Omega(t)$, which to our knowledge has not previously been used in the analysis of surface finite element methods.
\begin{definition}
\label{defn: new L2 projection}
For $z_h \in L^2(\Omega_h(t))$ we define $\Ph z_h \in L^2(\Omega(t))$ to be the unique function such that
\begin{align}
    m(\Ph z_h, \phi) = m_h(z_h, \phi^{-\ell}) \quad \forall \phi \in L^2(\Omega(t)). \label{eqn: new L2 projection defn}
\end{align}
\end{definition}
Clearly such a function exists by the Riesz representation theorem.
We note that this operator is not truly a projection since $L^2(\Omega(t))$ is not a subset of $L^2(\Omega_h(t))$ --- however, it is ``almost a projection'' as we shall see in Lemma~\ref{lemma: new L2 projection bounds}.
We now state some of the basic properties of this operator.

\begin{lemma}
    \label{lemma: new L2 projection bounds}
    Given $z_h \in L^2(\Omega_h(t))$, there exists a constant $C$ independent of $z_h$, $t$, and $h$ such that
    \begin{gather}
        \|\Ph z_h\|_{L^2(\Omega(t))} \leq C \|z_h\|_{L^2(\Omega_h(t))} \label{eqn: new L2 projection bound}\\
        \|\Ph z_h - z_h^\ell\|_{L^2(\Omega(t))} \leq C h^2 \|z_h\|_{L^2(\Omega_h(t))}. \label{eqn: new L2 projection error}
    \end{gather}
    Moreover, $\Ph$ is almost a projection in the sense that
    \begin{align}
        \|\Ph(\Ph z_h)^{-\ell} - \Ph z_h\|_{L^2(\Omega(t))} \leq Ch^2\|z_h\|_{L^2(\Omega_h(t))}. \label{eqn: L2 almost projection}
    \end{align}
\end{lemma}
\begin{proof}
    Proving~\eqref{eqn: new L2 projection bound} is a straightforward consequence of the stability of the lift after testing \eqref{eqn: new L2 projection defn} with $\phi = \Ph z_h$.
    In order to show~\eqref{eqn: new L2 projection error} we observe that
    \[ \|\Ph z_h -z_h^\ell\|_{L^2(\Omega(t))}^2 = m_h(z_h, (\Ph z_h - z_h^\ell)^{-\ell}) - m(z_h^\ell, \Ph z_h - z_h^\ell), \]
    whence using Lemma~\ref{lemma: geometric perturbations} yields the result.
    Finally, we verify the ``almost projection'' property~\eqref{eqn: L2 almost projection}.
    It is straightforward to see from~\eqref{eqn: new L2 projection defn} and the stability of the lift that
    \begin{align*}
        \|\Ph(\Ph z_h)^{-\ell} - \Ph z_h\|_{L^2(\Omega(t))}^2 &= m_h((\Ph z_h)^{-\ell} - z_h,(\Ph(\Ph z_h)^{-\ell} - \Ph z_h)^{-\ell})\\
        &\leq C \|\Ph z_h - z_h^\ell\|_{L^2(\Omega(t))}\|\Ph(\Ph z_h)^{-\ell} - \Ph z_h\|_{L^2(\Omega(t))},
    \end{align*}
    from~\eqref{eqn: L2 almost projection} now follows by applying~\eqref{eqn: new L2 projection error}.
\end{proof}

\begin{remark}
    \label{remark: projection-type operator determinant}
    By a straightforward change of variables it is easy to verify that
    \[ \Ph z_h = \det(D \mbf{p}^{-1}) z_h^\ell \quad \text{a.e. on } \Omega(t), \]
    where $\mbf{p}^{-1} : \Omega(t) \rightarrow \Omega_h(t)$ is the inverse of the closest point projection.
\end{remark}

The following technical result we prove concerns the differentiability in time of $\Ph z_h$ for sufficiently smooth (in time) finite element functions $z_h$.
\begin{lemma}
    \label{lemma: new L2 projection time derivative}
    Let $z_h \in \mathcal{S}_h$.
    Then $\Ph z_h$ as defined in Definition~\ref{defn: new L2 projection} is an element of $H^1_{H^{-1}}$.
    Moreover, there exists a constant, $C$, independent of $z_h$ and $h$, such that
    \begin{gather}
        \int_0^T \|\matdev \Ph z_h\|_{H^{-1}(\Omega(t))}^2 \leq C \int_0^T \left( \|\matdev_h z_h\|_{H^{-1}(\Omega_h(t))}^2 + \|z_h\|_{L^2(\Omega_h(t))}^2 \right), \label{eqn: new L2 projection time derivative bound}\\
        \int_0^T \left| m_*(\matdev \Ph z_h, \phi) - m_h(\matdev_h z_h, \phi^{-\ell}) \right| \leq Ch^2 \int_0^T \|z_h\|_{L^2(\Omega_h(t))}\|\phi\|_{H^1(\Omega(t))}. \label{eqn: new L2 projection time derivative error}
    \end{gather}
\end{lemma}
\begin{proof}
    The outline of this technical result is as follows: $z_h \in \mathcal{S}_h$ implies that the right-hand side of~\eqref{eqn: new L2 projection defn} is differentiable in time; we can use this to find a candidate for $\matdev \Ph z_h$ in the sense of \cite[Definition 2.28]{Alphonse2015Abstract}; finally, with this notion of derivative one can use~\eqref{eqn: new L2 projection error} and Lemma~\ref{lemma: geometric perturbations} to obtain~\eqref{eqn: new L2 projection time derivative error}.
    
    Formally, by differentiating~\eqref{eqn: new L2 projection defn} in time and using Lemma~\ref{lemma: transport theorem} and Lemma~\ref{lemma: discrete transport theorem} one obtains, for~$\phi \in H^1_{H^1}$,
    \[ m_*(\matdev \Ph z_h, \phi) + m(\Ph z_h, \matdev \phi) + g(\Ph z_h, \phi) = m_h(\matdev_h z_h, \phi^{-\ell}) + m_h(z_h, \matdev_h \phi^{-\ell}) + g_h(z_h, \phi^{-\ell}). \]
    By using~\eqref{eqn: new L2 projection defn} and the definition of $\matdev_\ell$ one finds that
    \begin{equation}
    m_*(\matdev \Ph z_h, \phi) = m_h(\matdev_h z_h, \phi^{-\ell}) + m(\Ph z_h, \matdev_\ell \phi - \matdev \phi) + g_h(z_h, \phi^{-\ell}) - g(\Ph z_h, \phi),\label{eqn: L2 time derivative pf1}
    \end{equation}
    where one uses Lemma~\ref{lemma: time derivative difference} to see that this right-hand side is a linear functional acting on~$H^1(\Omega(t))$, and moreover all of the terms are known to exist.
    Hence we define a linear functional, $\widetilde{z}$, acting on $L^2_{H^1}$ by
    \[ \angled{\widetilde z, \phi}_{L^2_{H^{-1}} \times L^2_{H^1}} \coloneqq \int_0^T \left(m_h(\matdev_h z_h, \phi^{-\ell}) + m(\Ph z_h, \matdev_\ell \phi - \matdev \phi) + g_h(z_h, \phi^{-\ell}) - g(\Ph z_h, \phi) \right), \]
    which one then verifies (cf.~\cite[Definition 2.28]{Alphonse2015Abstract}) is the weak time derivative of $\Ph z_h$.
    We remark that we have abused notation here, since the functional~$\phi \mapsto \int_0^T m(\Ph z_h, (\matdev_\ell - \matdev)\phi)$ is not defined on $L^2_{H^1}$.
    However, owing to Lemma~\ref{lemma: time derivative difference}, one finds that for all $\phi \in H^1_{H^1}$
    \[ \left|\int_0^T m(\Ph z_h, \matdev_\ell \phi - \matdev \phi)\right| \leq Ch^2 \left(\int_0^T \|z_h\|_{L^2(\Omega_h(t))}^2\right)^\frac{1}{2} \left(\int_0^T \|\phi\|_{H^1(\Omega(t))}^2\right)^\frac{1}{2}, \]
    whence one may observe that by the dense embedding $H^1_{H^1} \overset{d}{\hookrightarrow} L^2_{H^1}$ we may uniquely extend this functional to act on elements of $L^2_{H^1}$.
    One can readily obtain~\eqref{eqn: new L2 projection time derivative bound} from the definition of $\widetilde{z}$.
    The error bound~\eqref{eqn: new L2 projection time derivative error} is also straightforward to show from~\eqref{eqn: L2 time derivative pf1}.
    By noting that
    \begin{align*}
        |g_h(z_h, \phi^{-\ell}) - g(\Ph z_h, \phi)| \leq |g_h(z_h, \phi^{-\ell}) - g(z_h^\ell, \phi)| + |g(z_h^\ell - \Ph z_h, \phi)|,
    \end{align*}
    one may use \eqref{eqn: new L2 projection error}, Lemma~\ref{lemma: time derivative difference} and Lemma~\ref{lemma: geometric perturbations} to obtain~\eqref{eqn: new L2 projection time derivative error} after integrating~\eqref{eqn: L2 time derivative pf1} in time.
\end{proof}

\begin{remark}
    \label{remark: new l2 projection}
    The motivation for this projection is twofold.
    Firstly, this provides a notion of a lift such that functions defined on $\Omega_h(t)$ with mean value $0$ are mapped to functions on $\Omega(t)$ with mean value $0$ --- hence this will be compatible with the use of inverse Laplacians.
    Secondly, due to Lemma~\ref{lemma: new L2 projection time derivative}, this projection allows one to mitigate any issues to do with the different material derivatives, $\matdev$ and $\matdev_h$.
    In particular this means that our analysis does not require $L^2_{H^2}$ regularity of $\matdev e$, as is often required in the error analysis of ESFEM~\cite{Elliott2021EvolvingFiniteElement}, which one typically does not have for the Stefan problem.
    As such we believe the ideas used here will be useful in the analysis of ESFEM for singular/degenerate PDEs with limited regularity, such as the porous medium equation~\cite{Alphonse2016Porous}; the Cahn--Hilliard equation~\cite{Elliott2026Degenerate,Elliott2026Logarithmic}; and the parabolic p-Laplace equation~\cite[Example 7.3]{Alphonse2023EvolvingBanach}\cite{Miura2026pLaplacian,Miura2026NeumannPLaplacian}.
\end{remark}

\subsubsection{Error analysis}
To now derive suitable error equations let us introduce shorthand notation
\[ \Eh \coloneqq e - \Ph e_h. \]
We note that we may test \eqref{eqn: stefan problem2} with $\phi = \mathcal{G}\Eh$, which we observe is well-defined since
\[ \int_{\Omega(t)} e(t) = \int_{\Omega(0)} e_0 = \int_{\Omega_h(0)} e_h^0 = \int_{\Omega_h(t)} e_h(t) = \int_{\Omega(t)} \Ph e_h(t). \]
Likewise we may test we test \eqref{eqn: semidiscrete stefan} with $\phi_h = \mathcal{G}_h \Eh$ which is well-defined for the same reason.
Doing this one obtains
\begin{gather}
    m_*(\matdev e, \mathcal{G}\Eh) + g(e, \mathcal{G}\Eh) + m(\mathcal{U}(e), \Eh) = m(f, \mathcal{G}\Eh), \label{eqn: error pf1}\\
    m_h(\matdev_h e_h, \mathcal{G}_h\Eh) + g_h(e_h, \Gh \Eh) + m_h\left(R_h\mathcal{U}(e_h), \Eh^{-\ell} - \mval{\Eh^{-\ell}}{\Omega_h(t)}\right) = m_h(f_h, \mathcal{G}_h\Eh), \label{eqn: error pf2}
\end{gather}
where we have used the Definition~\ref{defn: inverse laplace Omega}, Definition~\ref{defn: inverse laplace Omegah} and Definition~\ref{defn: ritz projection}.
We now rewrite \eqref{eqn: error pf2} in terms of $\Ph e_h$ as
\begin{align}
        m_*(\matdev \Ph e_h, \Gh^\ell \Eh) + g(\Ph e_h, \Gh^\ell\Eh) + m(\mathcal{U}(\Ph e_h),\Eh) = m_h(f_h, \mathcal{G}_h \Eh) + \sum_{n=1}^5 I_n
    \label{eqn: error pf3}
\end{align}
where we have defined consistency errors
\begin{gather*}
    I_1 \coloneqq m_*(\matdev \Ph e_h, \Gh^\ell \Eh ) - m_h(\matdev_h e_h, \Gh \Eh),\\
    I_2 \coloneqq g(\Ph e_h, \Gh^\ell \Eh) - g_h(e_h, \Gh \Eh),\\
    I_3 \coloneqq m(\mathcal{U}(\Ph e_h), \Eh) - m_h(\mathcal{U}(e_h), \Eh^{-\ell}),\\
    I_4 \coloneqq m_h(\mathcal{U}(e_h) - R_h \mathcal{U}(e_h), \Eh^{-\ell}),\\
    I_5 \coloneqq \left( \int_{\Omega_h(t)} \mathcal{U}(e_h) \right)\left( \mval{\Eh^{-\ell}}{\Omega_h(t)} \right),
\end{gather*}
so that when we subtract \eqref{eqn: error pf3} from \eqref{eqn: error pf1} we obtain
\begin{align}
    m_*(\matdev \Eh, \mathcal{G}\Eh) + g(\Eh, \mathcal{G}\Eh) + m(\mathcal{U}(e) - \mathcal{U}(\Ph e_h), \Eh) = \sum_{n = 1}^8 I_n, \label{eqn: error pf4}
\end{align}
where we have introduced further consistency errors defined by
\begin{gather*}
    I_6 \coloneqq m_*(\matdev \Ph e_h, \mathcal{G} \Eh - \mathcal{G}_h^\ell \Eh),\\
    I_7 \coloneqq g(\Ph e_h, \mathcal{G} \Eh - \mathcal{G}_h^\ell \Eh),\\
    I_8 \coloneqq m(f , \mathcal{G}\Eh) - m_h(f_h, \mathcal{G}_h \Eh).
\end{gather*}
The main component of our error analysis is the following lemma concerning these consistency errors.

\begin{lemma}[Consistency errors]
\label{lemma: consistency errors}
    Let $\beta$ satisfy Assumption~\ref{assumption: enthalpy}, and let $e_h^0$ and $f_h$ satisfy Assumption~\ref{assumption: data}.
    For $I_1, \ldots, I_8$ as defined above one has
    \begin{align*}
        \int_0^T |I_1| &\leq C_1 h^2 \int_0^T \|e_h\|_{L^2(\Omega_h(t))}\|\Gh^\ell \Eh\|_{H^1(\Omega(t))},\\
        \int_0^T |I_2| &\leq  C_2 h \int_0^T \|e_h\|_{L^2(\Omega_h(t))}\|\Gh^\ell \Eh\|_{L^2(\Omega(t))}, \\
        \int_0^T |I_3| &\leq C_3 h^2 \int_0^T \|e_h\|_{L^2(\Omega_h(t))} \|\mathcal{E}_h^{-\ell}\|_{L^2(\Omega(t))},\\
        \int_0^T |I_4| &\leq C_4 h \int_0^T \|\mathcal{U}(e_h)\|_{H^1(\Omega_h(t))} \|\Eh\|_{L^2(\Omega(t))},\\
        \int_0^T |I_5| &\leq C_5 h^2 \int_0^T \|e_h\|_{L^1(\Omega_h(t))} \|\Eh\|_{L^2(\Omega(t))},\\
        \int_0^T |I_6| &\leq C_6 h \left(\|e_h^0\|_{L^2(\Omega_h(0))}^2 + \int_0^T \|f_h\|_{L^2(\Omega_h(t))}^2 \right)^{\frac{1}{2}} \left( \int_0^T\|\Eh\|_{L^2(\Omega(t))}^2\right)^{\frac{1}{2}},  \\
        \int_0^T |I_7| &\leq C_7 h^2 \int_0^T \|e_h\|_{L^2(\Omega_h(t))} \|\Eh\|_{L^2(\Omega(t))},\\
        \int_0^T |I_8| &\leq C_8 \int_0^T \|f - f_h^\ell\|_{L^2(\Omega_h(t))} \|\Eh\|_{-1,t} + h^2\|f_h\|_{L^2(\Omega_h(t))} \|\Eh\|_{L^2(\Omega(t))},
    \end{align*}
    where the constants are independent of $h, e, e_h$.
    Note also that $C_3$ and $C_5$ depend linearly on $C_\mathcal{U}$, the Lipschitz constant of $\mathcal{U}(\cdot)$.
\end{lemma}
\begin{proof}
    It is clear that the bound for $I_1$ is an immediate consequence of Lemma~\ref{lemma: new L2 projection time derivative}.
    The bound for $I_2$ is a consequence of combining \eqref{eqn: geometric perturbation g1}, \eqref{eqn: geometric perturbation g2} and Lemma~\ref{lemma: new L2 projection bounds}.
    The bound for $I_3$ follows by writing
    \[ m(\mathcal{U}(\Ph e_h), \Eh) - m_h(\mathcal{U}(e_h), \Eh^{-\ell}) = m(\mathcal{U}(\Ph e_h) - \mathcal{U}(e_h^\ell), \Eh) + m(\mathcal{U}(e_h^\ell), \Eh) - m_h(\mathcal{U}(e_h), \Eh^{-\ell}), \]
    and combining this with \eqref{eqn: geometric perturbation m}, Lemma~\ref{lemma: new L2 projection bounds}, and the Lipschitz continuity of $\mathcal{U}$.
    $I_4$ is bounded as an immediate consequence of Lemma~\ref{lemma: ritz error}.
    One bounds $I_5$ by noting that, since $\int_{\Omega(t)} \Eh = 0$,
    \[ \mval{\Eh^{-\ell}}{\Omega_h(t)} = \frac{1}{|\Omega_h(t)|} \left( \int_{\Omega_h(t)} \Eh^{-\ell} - \int_{\Omega(t)} \Eh \right), \]
    and then by using~\eqref{eqn: geometric perturbation m} and the Lipschitz continuity of $\mathcal{U}(\cdot)$.
    The bound on $I_6$ follows by using Lemma~\ref{lemma: inverse laplacian error} along with~\eqref{eqn: semidiscrete energy estimate2},~\eqref{eqn: new L2 projection time derivative bound} and H\"older's inequality.
    One bounds $I_7$ by combining Lemma~\ref{lemma: inverse laplacian error} and \eqref{eqn: new L2 projection bound}.
    Finally to bound $I_8$ one writes
    \begin{align*}
        m(f, \G \Eh) - m_h(f_h, \Gh \Eh) &= m(f - f_h^\ell, \G \Eh) + m(f_h^\ell, \G \Eh - \Gh^\ell \Eh)\\
        &+ m(f_h^\ell, \Gh^\ell \Eh) - m_h(f_h, \Gh \Eh),
    \end{align*}
    and bounding these terms by using Lemma~\ref{eqn: inverse laplacian error},~\eqref{eqn: geometric perturbation m}, and the Poincar\'e inequality where necessary.
\end{proof}

\begin{remark}
    From Lemma~\ref{lemma: consistency errors} it is clear that the bottleneck in our error is due to our bounds for $I_4$ and $I_6$.
    Ultimately these will lead the $\mathcal{O}(\sqrt{h})$ error in the following result.
\end{remark}

We now state, and prove, our main result.
\begin{theorem}
\label{thm: error theorem}
	Let $\beta$ satisfy Assumption~\ref{assumption: enthalpy}, and let $e_h^0$ and $f_h$ satisfy Assumption~\ref{assumption: data}.
    Then if $e$ denotes the unique solution of~\eqref{eqn: weak stefan} and $e_h$ denotes the unique solution of~\eqref{eqn: semidiscrete stefan} one has
    \begin{multline}
        \|e - e_h^\ell\|_{H^{-1}(\Omega(T))}^2 + 2C_\beta \int_0^T \|\mathcal{U}(e) - \mathcal{U}(e_h^\ell)\|_{L^2(\Omega(t))}^2\\
        \leq \widehat{C}\left(h + \|e_0 - \Ph e_h^0\|_{-1,0}^2 + \int_0^T \|f - f_h^\ell\|_{L^2(\Omega(t))}^2 \right), \label{eqn: error theorem bound}
    \end{multline}
    where $\mathcal{P}_h$ is as defined in Definition~\ref{defn: new L2 projection}, and $\widehat{C}$ is a constant independent of $h$, but depending on $T$, $e$, $C_\mathcal{U}$, and $C_\beta$.
\end{theorem}
\begin{proof}
We recall that we have defined~$\Eh \coloneqq e - \Ph e_h$.
We use Lemma~\ref{lemma: transport theorem} on the first two terms in \eqref{eqn: error pf3} to see that
\[ m_*(\matdev \Eh, \mathcal{G}\Eh) + g(\Eh, \mathcal{G}\Eh) = \frac{1}{2} \ddt{} \|\Eh\|_{-1,t}^2 + \frac{1}{2} b(\mathcal{G}\Eh, \mathcal{G}\Eh). \]
By combining this with by using \eqref{eqn: beta strong monotonicity} one finds that \eqref{eqn: error pf4} yields
\begin{align}
    \frac{1}{2} \ddt{} \|\Eh\|_{-1,t}^2 + C_\beta \|\mathcal{U}(e) - \mathcal{U}(\Ph e_h)\|_{L^2(\Omega(t))}^2 \leq -\frac{1}{2} b(\mathcal{G}\Eh, \mathcal{G}\Eh) + \sum_{n=1}^8 |I_n|. \label{eqn: error pf5}
\end{align}
From the smoothness assumptions on $\mbf{V}$ one finds that
\[ |b(\mathcal{G}\Eh, \mathcal{G}\Eh)| \leq C \|\Eh\|_{-1,t}^2, \]
and we use this bound and integrate in time for 
\begin{align}
    \|\Eh\|_{-1,T}^2 + 2C_\beta \int_0^T \|\mathcal{U}(e) - \mathcal{U}(\Ph e_h)\|_{L^2(\Omega(t))}^2 \leq \|\Eh(0)\|_{-1,0}^2 + C \int_0^T \|\Eh\|_{-1,t}^2 + 2\sum_{n=1}^8 \int_0^T |I_n| . \label{eqn: error pf6}
\end{align}
One now uses Lemma~\ref{lemma: consistency errors} and Young's inequality to see that
\begin{align}
    \|\Eh\|_{-1,T}^2 + 2C_\beta \int_0^T \|\mathcal{U}(e) - \mathcal{U}(\Ph e_h)\|_{L^2(\Omega(t))}^2 \leq C\left(h + \|\Eh(0)\|_{-1,0}^2 + \int_0^T \|f - f_h^\ell\|_{L^2(\Omega(t))}^2 \right), \label{eqn: error pf7}
\end{align}
for a constant $C$ independent of $h$, but depending on $T$ and $C_\mathcal{U}$.
The final step is now to replace $\Ph e_h$ with $e_h^\ell$, for which we firstly recall that
\[ \|\Eh\|_{H^{-1}(\Omega(t))} \leq \| \Eh\|_{-1,t}. \]
We use this, and~\eqref{eqn: new L2 projection error}, to compute
\begin{equation}
\begin{aligned}
\|e - e_h^\ell\|_{H^{-1}(\Omega(T))} &= \sup_{\phi \in H^1(\Omega(T))} \frac{m(e - e_h^\ell, \phi)}{\|\phi\|_{H^1(\Omega(T))}} \leq \|\Eh\|_{-1,T} + \sup_{\phi \in H^1(\Omega(T))} \frac{m(\Ph e - e_h^\ell, \phi)}{\|\phi\|_{H^1(\Omega(T))}}\\
& \leq \|\Eh\|_{-1,T}^2 + Ch^2 \|e_h\|_{L^2(\Omega(T))}^2.
\end{aligned}
\label{eqn: error pf8}
\end{equation}
Similarly one can use the Lipschitz continuity of $\mathcal{U}(\cdot)$, and~\eqref{eqn: new L2 projection error} to see that
\begin{align}
    \notag
    \int_{0}^T \|\mathcal{U}(e) - \mathcal{U}(e_h^\ell)\|_{L^2(\Omega(t))}^2 &\leq \int_0^T \|\mathcal{U}(e) - \mathcal{U}(\Ph e_h)\|_{L^2(\Omega(t))}^2 + \int_0^T \|\mathcal{U}(\Ph e_h) - \mathcal{U}(e_h^\ell)\|_{L^2(\Omega(t))}^2\\
    &\leq \int_0^T \|\mathcal{U}(e) - \mathcal{U}(\Ph e_h)\|_{L^2(\Omega(t))}^2 + \widetilde{C}h^2 \int_0^T \| e_h \|_{L^2(\Omega_h(t))}^2
\label{eqn: error pf9}
\end{align}
where $\widetilde{C}$ depends linearly on $C_{\mathcal{U}}$.
The bound~\eqref{eqn: error theorem bound} follows by combining~\eqref{eqn: error pf7},~\eqref{eqn: error pf8} and~\eqref{eqn: error pf9}.
\end{proof}

\begin{remark}
    \label{remark: convergence_order_comparison}
    The error bound we obtain here is of the same order as that in the stationary, flat case, cf.~\cite[Theorem 3.3]{Elliott1987StefanError} and \cite[Corollary 1]{Nochetto1988Degenerate}.
\end{remark}

\section{Numerical examples}
\label{section: numerics}
\subsection{Numerical methods}
In this section we consider the discretisation in time of \eqref{eqn: semidiscrete stefan} by a backward Euler time discretisation, similar to that of~\cite{Dziuk2012FullyDiscreteESFEM}.
In the following we consider a uniform timestep size $\tau = \frac{T}{N_T}$ for some $N_T \in \mbb{N}$.
We also introduce some shorthand notation for the current time, $t_n \coloneqq n \tau$, and the space of finite element functions at time $t_n$, $S_h^n \coloneqq S_h(t_n)$.
The fully discrete problem is now as follows:
Given data $e_h^{n-1} \in S_h^{n-1}$ and $f_h^n \in S_h^n$, find $e_h^n \in S_h^n$ such that
\begin{align}
	\frac{1}{\tau} \left(m_h(t_n; e_h^n, \phi_h^n) - m_h(t_{n-1}; e_h^{n-1}, \underline{\phi_h^n})\right) + a_h(t_n;\mathcal{U}(e_h^n), \phi_h^n) = m_h(t_n; f_h^n, \phi_h^n), \label{eqn: fullydiscrete stefan}
\end{align}
for all $\phi_h^n \in S_h^n$.
Here $\underline{\phi_h^n} \in S_h^{n-1}$ denotes the vector with the same nodal values as $\phi_h^n \in S_h^n$ but defined over the previous surface.
We do not analyse this fully discrete numerical method, but we expect that the analysis will follow by combining techniques introduced in this present work with those developed for fully discrete ESFEM, cf.~\cite{Dziuk2012FullyDiscreteESFEM,Elliott2025FullyDiscrete,Elliott2026Logarithmic,Kovacs2016,Lubich2013BackwardDifference}.
We now introduce two numerical methods to approximate \eqref{eqn: fullydiscrete stefan}.
\subsubsection{A method using quadrature}
For our first numerical method, we introduce a quadrature rule for the nonlinear term, as we did in Lemma~\ref{lemma: numerical integration comparison}.
In matrix-vector form, the fully discrete scheme \eqref{eqn: fullydiscrete stefan} with a quadrature rule may be written as
\begin{align}
	M^n \codeletter{e}^n + \tau A^n \mathcal{U}(\codeletter{e}^n) = M^{n-1} \codeletter{e}^{n-1} + \tau M^n \codeletter{f}^n. \label{eqn: matrix vector form quadrature}
\end{align}
where $\codeletter{e}^n$ and $\codeletter{f}^n$ denote the vector of nodal values for $e_h^n \in S_h^n$ and $f_h^n \in S_h^n$ respectively, and the mass and stiffness matrices are given by components
\[ M^n_{ij} = m_h(t_n; \phi_i^n, \phi_j^n), \quad A^n_{ij} = a_h(t_n; \phi_i^n, \phi_j^n), \]
for $\phi_i^n$ the `$i$'th basis function of $S_h^n$.
We implement this method in DUNE~\cite{Dedner2010DUNE}, solving the nonlinear problem by using a non-smooth Newton method~\cite{Qi1993NonsmoothNewton} where the corresponding linear problems are solved with an exact solver and our tolerance for the Newton solver is chosen as $10^{-7}$.

\subsubsection{A method using an exact discretisation}
We now consider a new approach to the discretisation of this problem, based on writing~\eqref{eqn: stefan problem2} as
\begin{align*}
	\matdev e + e (\grado \cdot \mbf{V}) - \grado \cdot( \mathcal{U}'(e) \grado e) = f.
\end{align*}

In matrix-vector form this yields a system of the form
\begin{align}
	M^n \codeletter{e}^n + \tau \mathcal{A}(t_n;\codeletter{e}^n)\codeletter{e}^n = M^{n-1} \codeletter{e}^{n-1} + \tau M^n \codeletter{f}^n, \label{eqn: matrix vector form exact}
\end{align}
where we have defined the solution-dependent matrix
\[  \mathcal{A}(t_n;\codeletter{e}^n)_{ij} = \int_{\Omega_h(t_n)} \mathcal{U}'(e_h^n) \gradoh \phi_i \cdot \gradoh \phi_j. \]
To construct this matrix for general functions $e_h^n$, or $\mathcal{U}$ may be challenging.
In the case that $e_h^n$ is piecewise linear, one may divide each of the simplices into smaller regions in which the integration of $\mathcal{U}'(e_h^n)$ follows via simple quadrature.
For example, when considering $\mathcal{U}$ as given in \eqref{eq:classicalU}, a piecewise polynomial function, one divides each of the simplices into the polygons/polyhedra in which  $e_h^n<0$ and $e_h^n\geq1$, in these regions, exact integration is the result of a classical quadrature rule.
This strategy extends to a wide class of functions $\mathcal{U}$ which are piecewise smooth and exact quadrature rules available on each smooth region.
We note that in applications it is typically the case that $\mathcal{U}$ is piecewise linear, cf.~\cite[Chapter 2]{Gupta2003Stefan}.

With this matrix assembled, the problem is solved by a fixed point iteration where we solve sequences of linear problems
\[ M^n \codeletter{e}^{n,k} + \tau \mathcal{A}(t_n;\codeletter{e}^{n,k-1})\codeletter{e}^{n,k} = M^{n-1} \codeletter{e}^{n-1} + \tau M^n \codeletter{f}^n, \]
with an exact solver, provided $\tau$ is sufficiently small.
We choose our initial guess as $\codeletter{e}^{n,0} = \codeletter{e}^{n-1}$, and choose our stopping criteria to be
\[|M^n \codeletter{e}^{n,k} + \tau \mathcal{A}(t_n;\codeletter{e}^{n,k})\codeletter{e}^{n,k} - M^{n-1} \codeletter{e}^{n-1} - \tau M^n \codeletter{f}^n| < \mathsf{tol},\]
where we set the tolerance for our fixed point solver to be $\mathsf{tol} = 10^{-7}$.

This method is somewhat similar to a recent method designed in~\cite{Peters2025FlagUpdates} where the authors solve a linear system determined by a ``flag'' indicating which phase the mesh point is in.
This approach also relies on the piecewise linear nature of $\mathcal{U}$.

Due to the fixed point nature of this algorithm, we found that this approach was slower than the approach using quadrature.
Our implementation of this fixed point method found that, for sufficiently small $\tau$, this method converged within a few iterations.
We leave the topic of more efficient solvers based on this approach for future work. 

\subsection{Examples on an evolving surface}

In this subsection we demonstrate some phenomena exhibited by the evolving surface Stefan problem which arise only on evolving domains.

\subsubsection{Nucleation in the absence of external heat sources}
\label{subsubsection: nucleation}
We now demonstrate a phenomenon which cannot happen on a stationary domain.
In the following we shall assume $f \equiv 0$.
Consider the graph $\beta$ given by
\begin{align}
    \beta(r) \coloneqq \begin{cases}
    \{r\}, & r < 1,\\
    [1,2], & r = 1,\\
    \{ r + 1\}, & r > 1,
\end{cases}
\label{eqn: nucleation_graph}
\end{align}
and initial data $e_0 \in L^\infty(\Omega(0))$ which takes values in $[1,1+\delta]$ for some $\delta \in (0,1)$.
On a stationary domain the maximum principle (see for instance~\cite[Corollary 10.4]{Brezis2011FunctionalAnalysis}) implies that the solution to~\eqref{eqn: stefan problem2} must be valued in $[1,1+\delta]$, and in particular the temperature is $u \equiv 1$ for all time.
However, on an evolving domain the usual maximum principle does not apply and indeed the surface evolution can now force the enthalpy to take arbitrary values in $\mbb{R}^+$.
To illustrate this if one has initial data as above, then testing~\eqref{eqn: weak stefan} with $\phi \equiv 1$ and using Lemma~\ref{lemma: transport theorem} one finds
\[ 0 = m_*(\matdev e ,1) + g(e,1) = \ddt{} \int_{\Omega(t)} e(t).\]
Hence one has that
\[ \int_{\Omega(t)} e(t) = \int_{\Omega(0)} e_0 \quad \text{for a.e. } t \in [0,T], \]
which one may rewrite as
\[ \mval{e(t)}{\Omega(t)} = \frac{|\Omega(0)|}{|\Omega(t)|} \mval{e_0}{\Omega(0)}. \]
In particular if the surface area, $|\Omega(t)|$, decreases enough so that there exists $t^* \in [0,T]$ such that
\[ |\Omega(t^*)| < \frac{1}{2}\int_{\Omega(0)} e_0, \]
then one finds
\[ \mval{e(t^*)}{\Omega(t^*)} = \frac{|\Omega(0)|}{|\Omega(t^*)|} \mval{e_0}{\Omega(0)} > 2. \]
Hence there must exist a region of positive measure on $\Omega(t^*)$ such that $e(t^*) > 2$, meaning that the surface evolution has caused the $\{u > 1\}$ phase to nucleate.
Conversely, if we consider an expanding surface, where there exists some $t^* \in [0,T]$ such that
\[ |\Omega(t^*)| > \int_{\Omega(0)} e_0, \]
then one finds that
\[ \mval{e(t^*)}{\Omega(t^*)} = \frac{|\Omega(0)|}{|\Omega(t^*)|} \mval{e_0}{\Omega(0)} < 1, \]
and hence the $\{u < 1\}$ phase has nucleated.

We demonstrate this phenomena in Figure~\ref{fig: phase_nucleation} in the following numerical example, solving~\eqref{eqn: matrix vector form quadrature}, on an ellipsoid given by the $0$-level set of
\[ \phi(x,y,z;t) = \frac{2x^2}{1 + \exp\left(\frac{3t}{2}\right)} + \frac{y^2}{(2 - \tanh(2t))^2} + z^2 - 1,\]
over a time interval $t \in [0,1]$, with initial data $e_0(x,y,z) \equiv  1.5$.
Here we consider a mesh with $h \approx 0.3878$, and a timestep size $\tau = 2\cdot 10^{-4}$.
The surface evolution here has no tangential component, and hence the nucleation of phases is due to geometric motion of the domain rather than advective effects.

\begin{figure}[ht]
    \centering
    \begin{subfigure}[b]{.8\linewidth}
        \begin{subfigure}[b]{.5\linewidth}
        \centering \includegraphics[width =\linewidth]{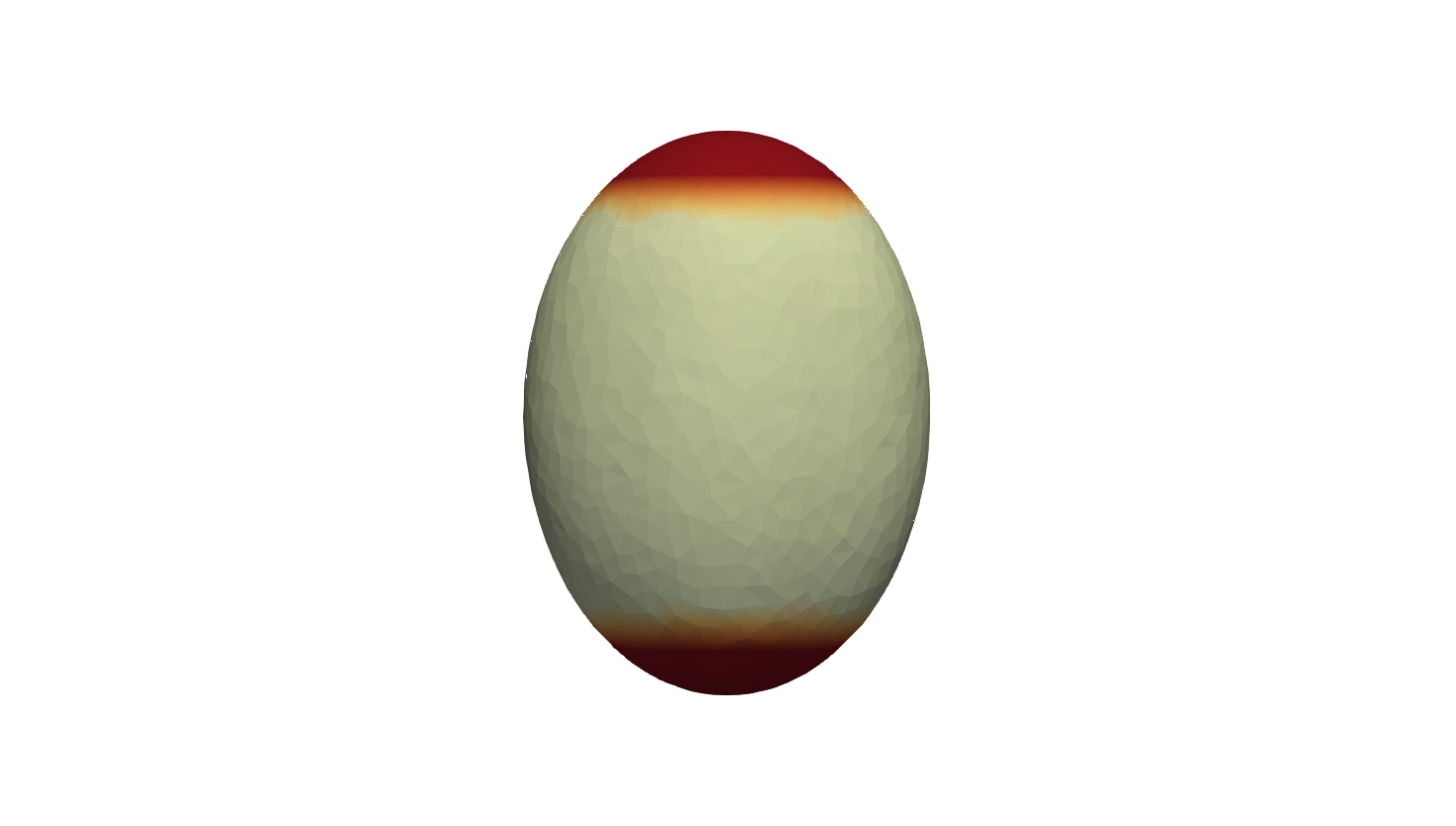}
        \caption{$t=0.25$.}
        \end{subfigure}
        ~
        \begin{subfigure}[b]{.5\linewidth}
        \centering \includegraphics[width =\linewidth]{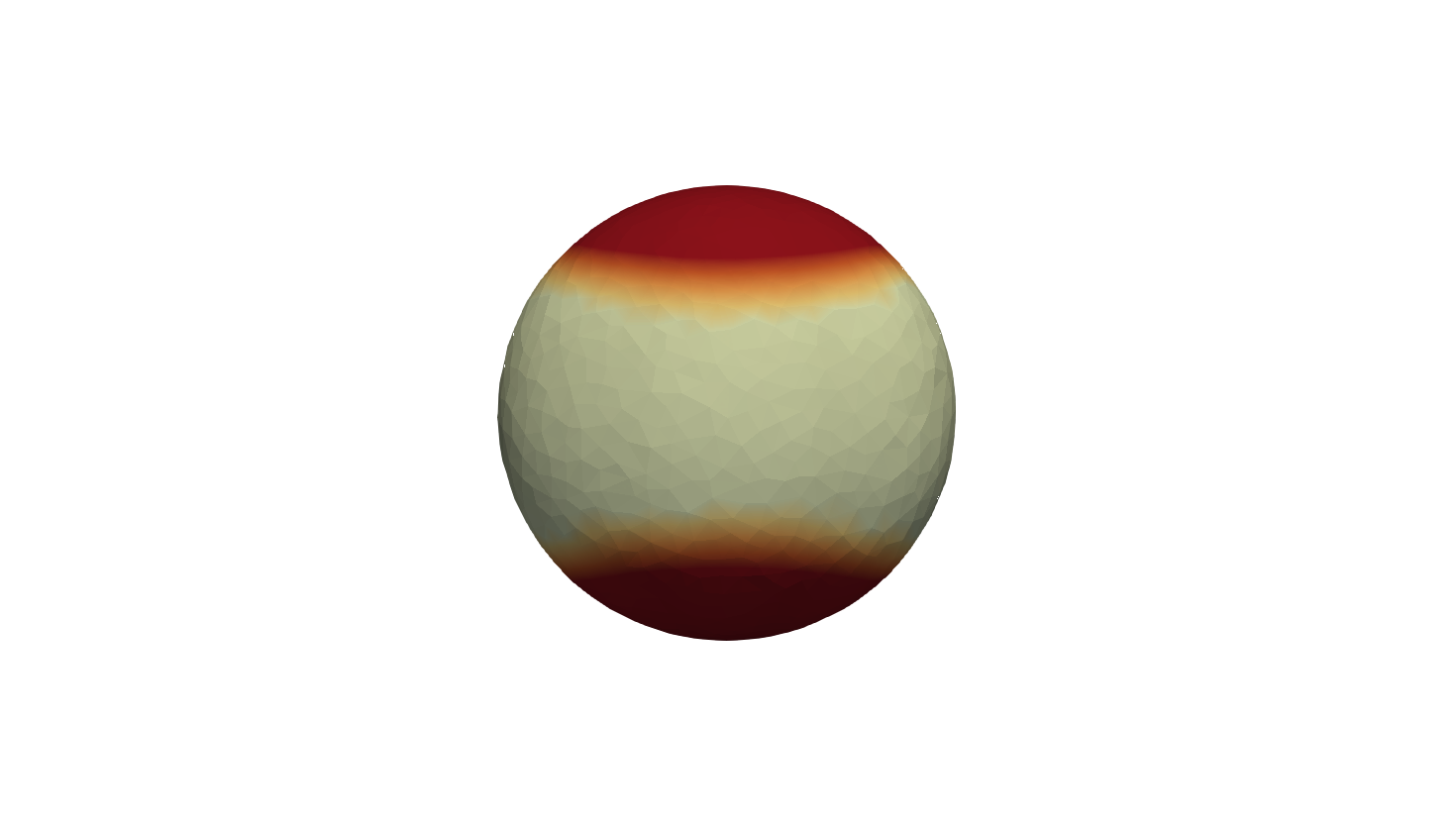}
        \caption{$t=0.5$.}
        \end{subfigure}
        \newline
        \begin{subfigure}[b]{.5\linewidth}
        \centering \includegraphics[width =\linewidth]{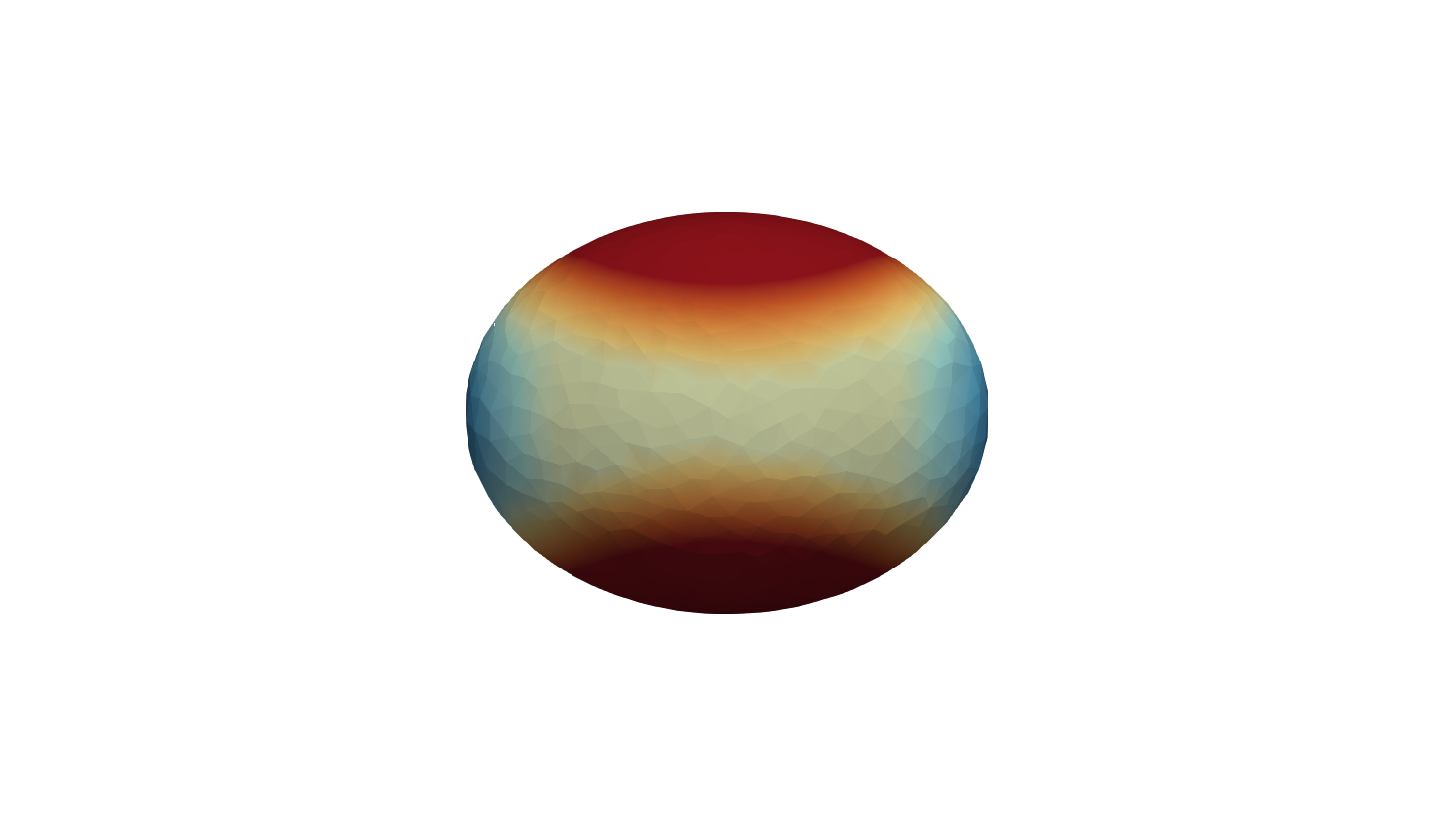}
        \caption{$t=0.75$.}
        \end{subfigure}
        ~
        \begin{subfigure}[b]{.5\linewidth}
        \centering \includegraphics[width =\linewidth]{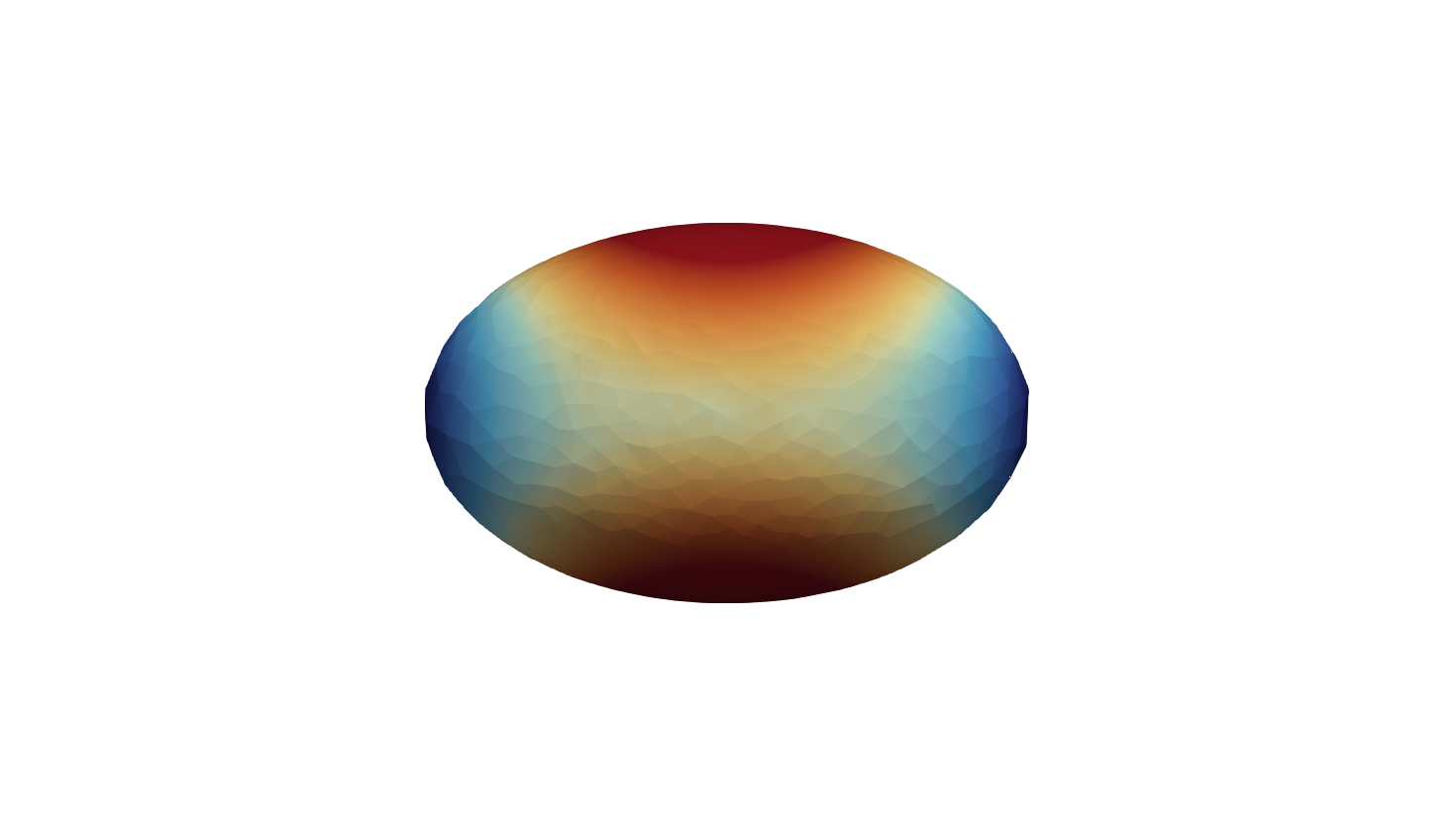}
        \caption{$t=1$.}
        \end{subfigure}
    \end{subfigure}
    ~
    \begin{subfigure}[b]{.15\linewidth}
            \centering
            \includegraphics[width = .5\linewidth]{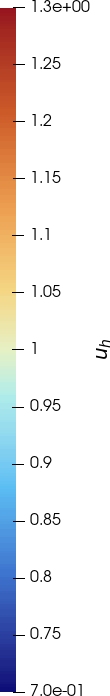}
    \end{subfigure}
    \caption{Plots of the temperature, $\mathcal{U}(e)$, indicating the nucleation of phases in the absence of external heat sources.
    Blue represents a region where the temperature is less than $1$, and red a region where the temperature is greater than $1$.
    Here we consider $\beta$ given by~\eqref{eqn: nucleation_graph}.}
    \label{fig: phase_nucleation}
\end{figure}

\subsubsection{Formation of mushy regions in the absence of external heat sources}
\label{subsubsection: mushy regions}
In the stationary, Euclidean setting it is known~\cite{Elliott1982MovingBoundaryProblems} that one may develop a mushy region in the presence of an external heat source.
Conversely, it is known~\cite{Nochetto1987NondegenerateStefanProblems,Rodrigues1989StefanRevisted}, that in the absence of external heat sources, and sufficiently smooth data, that mushy regions do not spontaneously develop.
However, on an evolving surface it appears to be the case that mushy regions can spontaneously develop even when one does not have an external heat source.
We demonstrate this in Figure~\ref{fig: stefan_mushy}, where we solve~\eqref{eqn: stefan problem2} with $f \equiv 0$, and choose $\beta$ given by~\eqref{eqn: enthalpy example}.
Heuristically this can be explained by thinking of the surface evolution of $\Omega(t)$ acting as an external heat source in some sense.
This implies that the existence of strong solutions, for which there cannot be a mushy region, to the Stefan problem on an evolving surface is quite delicate, and to our knowledge there are indeed no results of this kind.

In Figure~\ref{fig: stefan_mushy} we demonstrate the formation of such a mushy region in the absence of external heat sources.
Here we consider an evolving torus, with major radius $R(t) = 0.75 + 0.75t$ and minor radius $r(t) \equiv 0.25$, and initial data
\[ e_0(x,y,z) = \begin{cases}
    \tanh(10(x-0.4)), & x \leq 0.4,\\
    \tanh(10(x-0.4)) + 1, & x > 0.4.
\end{cases} \]
We consider a mesh with $h \approx 0.0772$ and a timestep size $\tau = 10^{-4}$.

\begin{figure}[ht]
    \centering
        \begin{subfigure}[b]{.8\linewidth}
        \begin{subfigure}[b]{.5\linewidth}
        \centering \includegraphics[width =\linewidth]{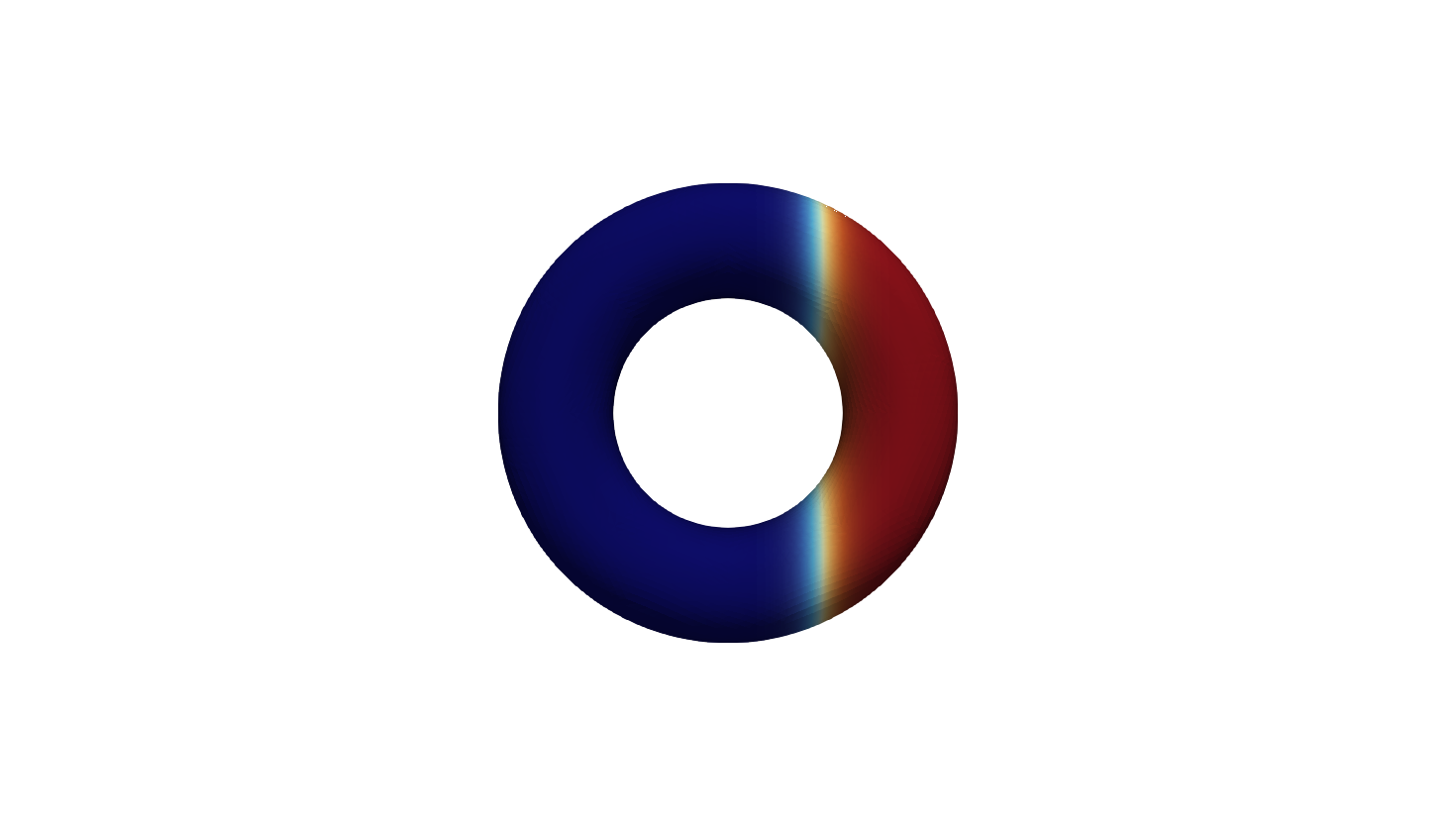}
        \caption{$t=0$.}
        \end{subfigure}
        ~
        \begin{subfigure}[b]{.5\linewidth}
        \centering \includegraphics[width =\linewidth]{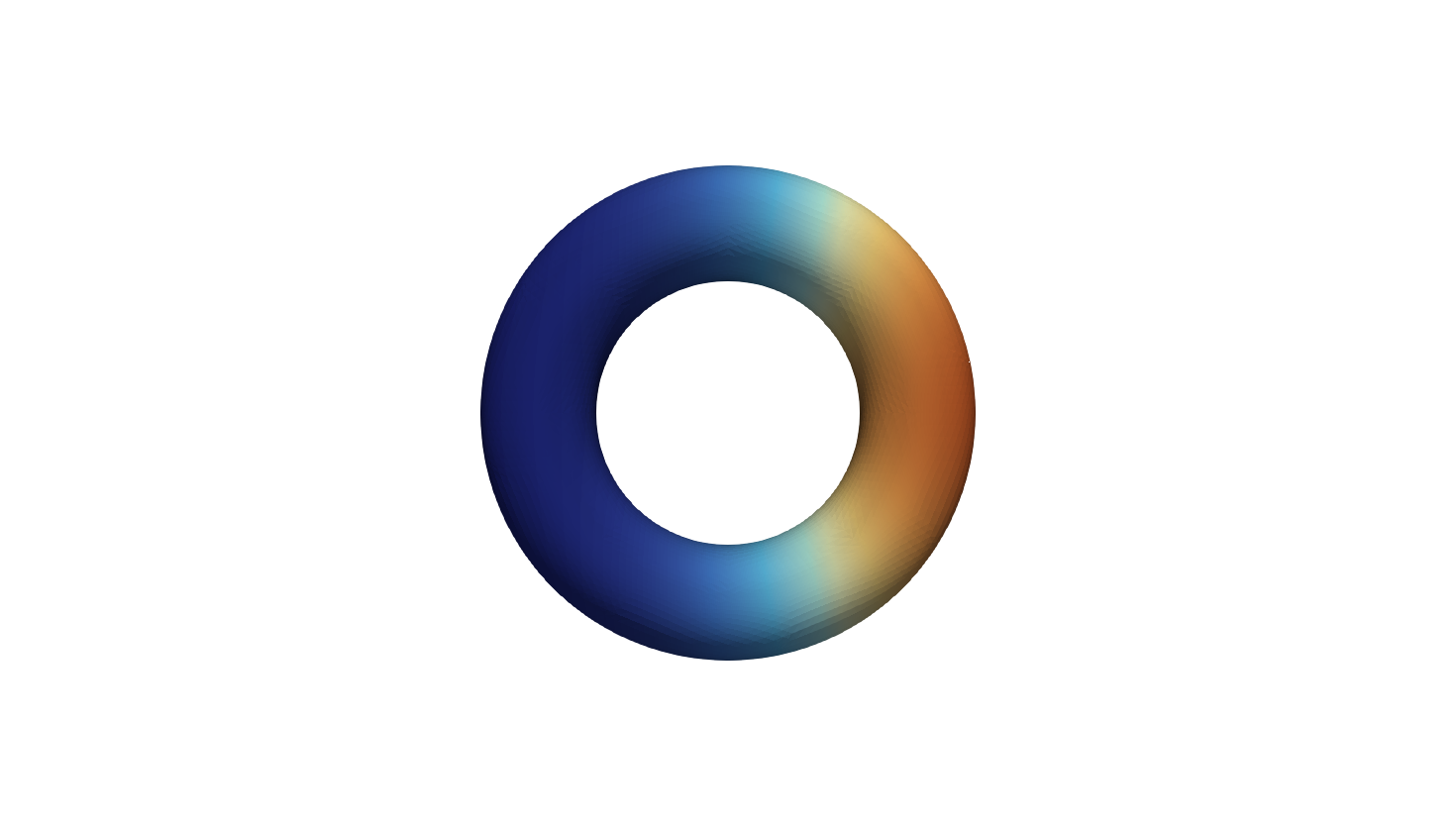}
        \caption{$t=0.1$.}
        \end{subfigure}
        \newline
        \begin{subfigure}[b]{.5\linewidth}
        \centering \includegraphics[width =\linewidth]{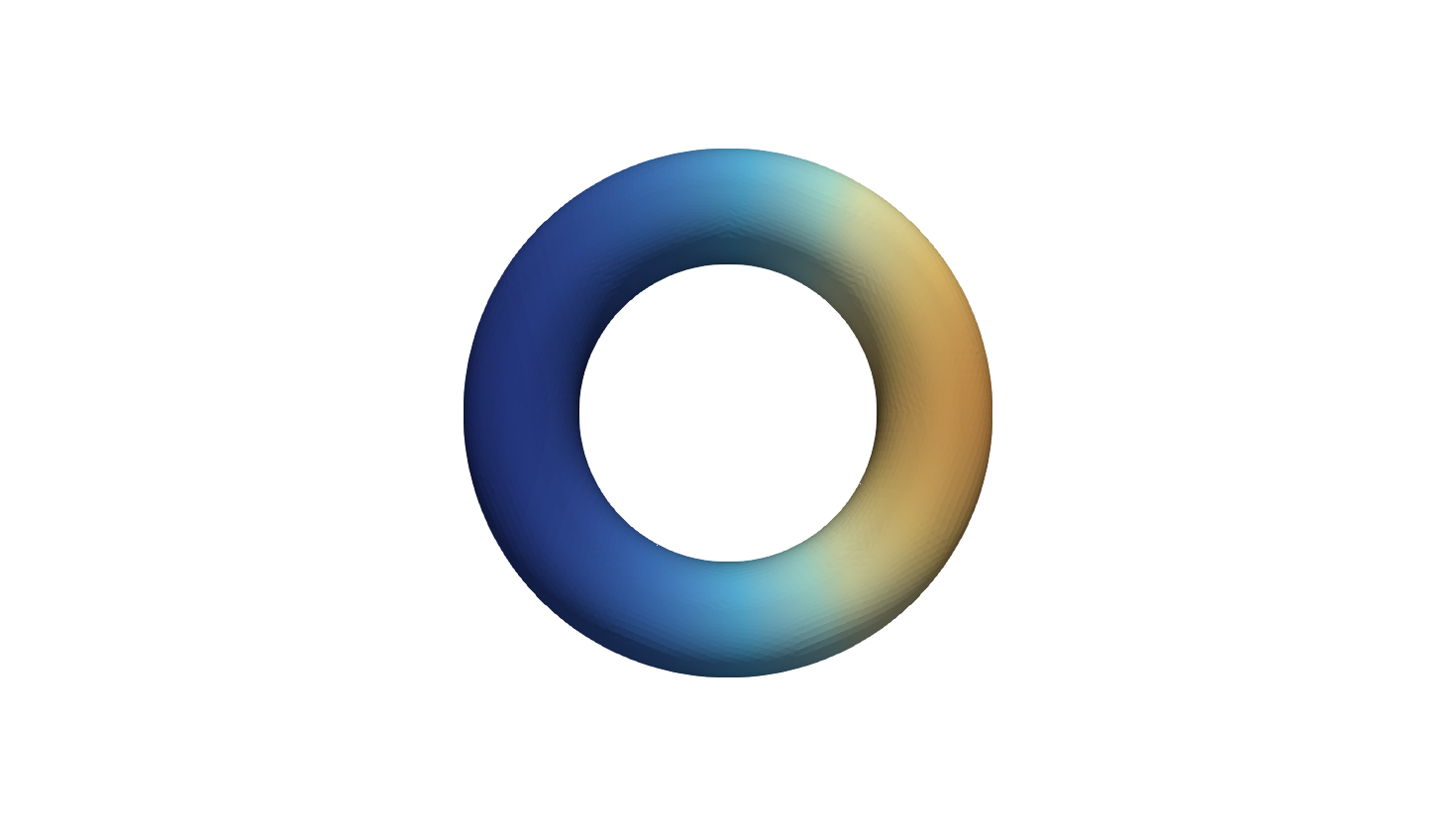}
        \caption{$t=0.2$.}
        \end{subfigure}
        ~
        \begin{subfigure}[b]{.5\linewidth}
        \centering \includegraphics[width =\linewidth]{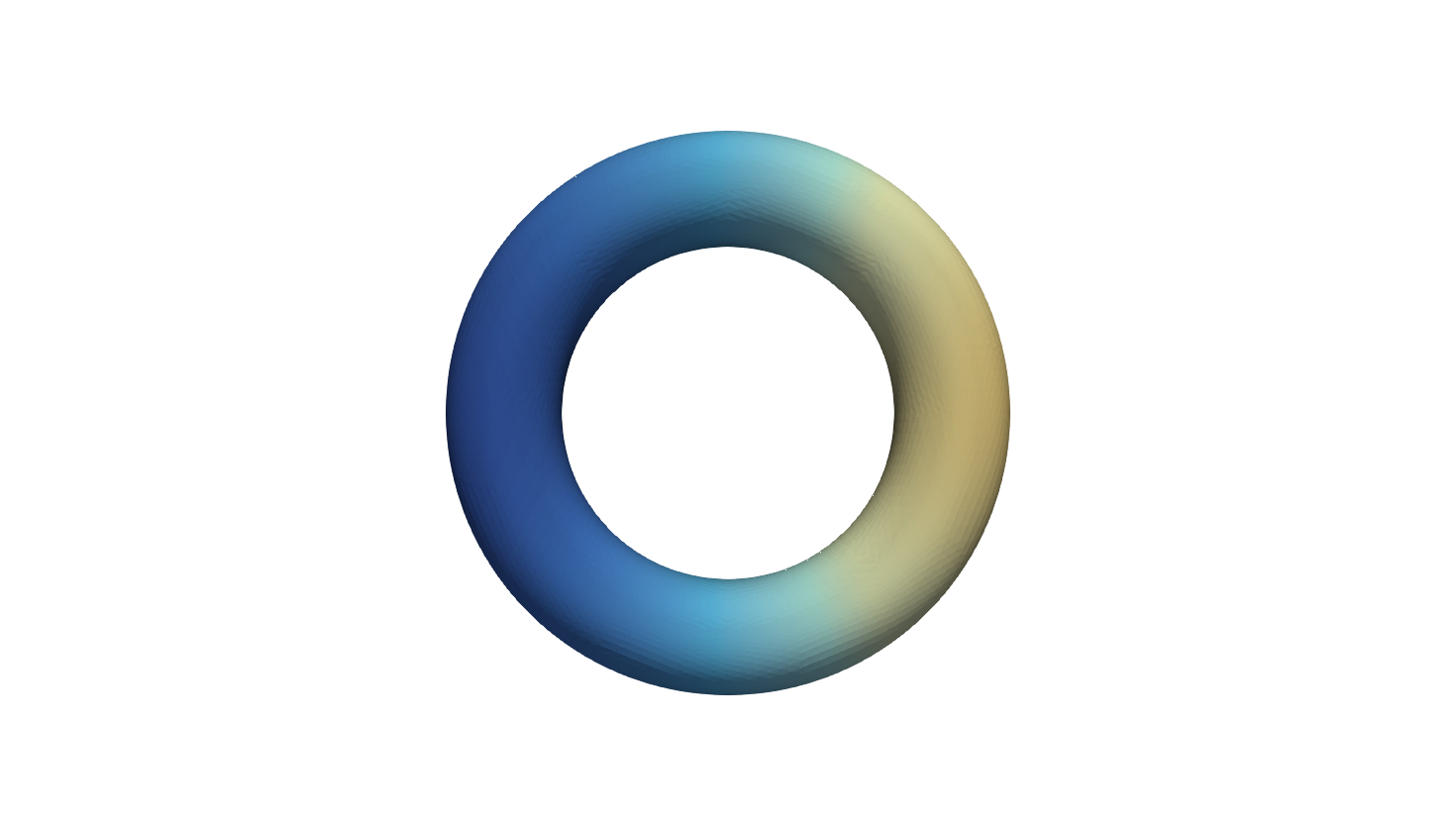}
        \caption{$t=0.3$.}
        \end{subfigure}
        \newline
        \begin{subfigure}[b]{.5\linewidth}
        \centering \includegraphics[width =\linewidth]{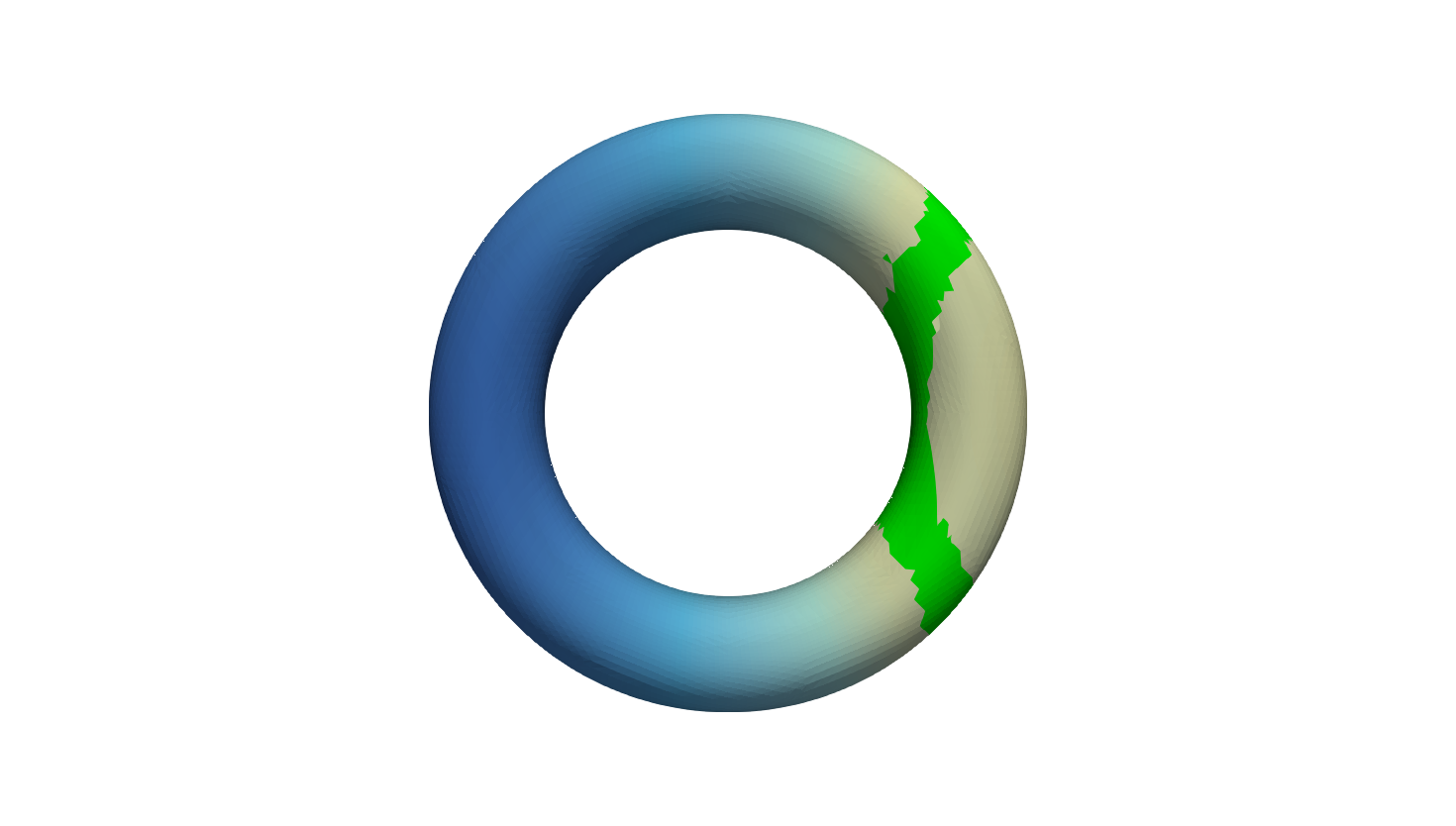}
        \caption{$t=0.4$.}
        \end{subfigure}
        ~
        \begin{subfigure}[b]{.5\linewidth}
        \centering \includegraphics[width =\linewidth]{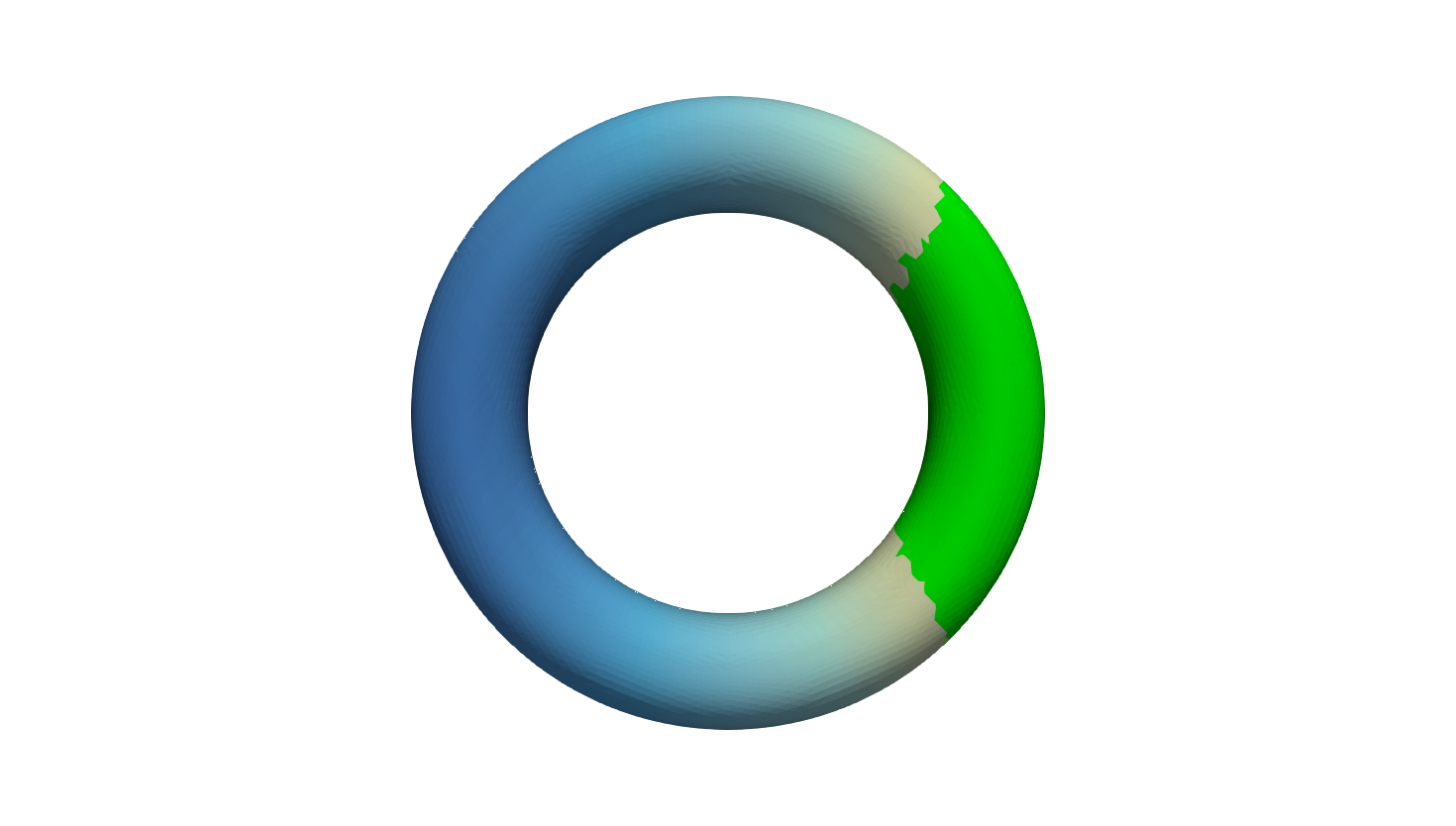}
        \caption{$t=0.5$.}
        \end{subfigure}
        \end{subfigure}
        ~
        \begin{subfigure}[b]{.15\linewidth}
            \centering
            \includegraphics[width = 1.1\linewidth]{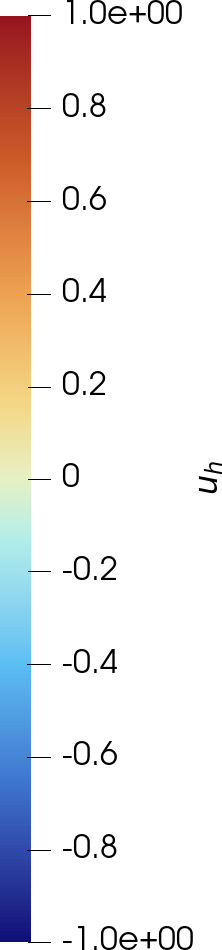}
        \end{subfigure}
    \caption{Plots of the temperature indicating the formation of a mushy region in the absence of external heat sources.
    Here blue represents a region where the temperature, $\mathcal{U}(e)$, is negative, red a region where the temperature is positive, and green a mushy region where enthalpy takes values in $[0,1]$.
    Here we consider $\beta$ given by~\eqref{eqn: enthalpy example}.
    }
    \label{fig: stefan_mushy}
\end{figure}

\subsection{Experimental order of convergence}
In the following we choose our timestep size to be $\tau = \mathcal{O}(h^2)$ so that the error we observe is dominated to the spatial discretisation, as considered in our analysis.
To approximate the $L^2_{L^2}$ norm of a function $z_h$ (with nodal vector $\codeletter{z}^n$ at $t = t_n$) we shall use the approximation
\[ \|z_h\|_{L^2_{L^2}} \approx \left( \sum_{n=1}^{N_T} \tau \|z_h^n\|_{L^2(\Omega_h(t))}^2 \right)^{\frac{1}{2}} = \left( \sum_{n=1}^{N_T} \tau \codeletter{z}^n \cdot M^n \codeletter{z}^n \right)^{\frac{1}{2}}.  \]
Similarly we approximate the $L^{\infty}_{H^{-1}}$ norm via
\[ \|z_h\|_{L^{\infty}_{H^{-1}}}^2 \approx \max_{n=1,\ldots,N_T} \left(\codeletter{z}^n \cdot M^n (M^n + A^n)^{-1}M^n \codeletter{z}^n  \right).\]
This right-hand expression can be viewed as a finite element approximation $m(t_n;z,\widetilde{\mathcal{G}}z) \approx \|z\|_{H^1(t_n)}^2$ for a suitably defined inverse Laplacian operator $\widetilde{\mathcal{G}}$.
One can then justify the use of this approximation by appealing to well-known error bounds for surface finite elements~\cite{Dziuk2013SurfacePDEs} --- we omit such justification.

Our experimental order of convergence (EOC) plots (Figure~\ref{fig: EOC_expanding_sphere} and Figure~\ref{fig: EOC_rotating_sphere}) indicate a higher order of convergence than that predicted by Theorem~\ref{thm: error theorem}.
This has previously been observed in~\cite{Nochetto1987NondegenerateStefanProblems}, where the authors observe $\mathcal{O}(h)$ convergence for the temperature in the $L^2_{L^2}$ norm.
For comparison, in our experiments we observe $\mathcal{O}(h^2)$ convergence for the temperature in the $L^2_{L^2}$ norm --- we expect this is due to our exact solution being quite smooth, as it is difficult to construct rougher closed-form solutions to a surface Stefan problem.

\subsubsection{Stationary interface on a expanding sphere}\label{sec:expSphere}
Here we consider our spatial domain to be an expanding sphere, given by the zero level set of
\[ \phi(x,y,z;t) = x^2 + y^2 + z^2 - e^{4t}, \]
and choose our initial data to be given by the Lagrange interpolant of $e_0$, where
\[ e_0(x,y,z) = \begin{cases}
    P_1(y) + \frac{1}{2}P_3(y), & y \leq 0,\\
    1 + P_1(y) + \frac{1}{2}P_3(y), & y > 0,
\end{cases} \]
where $P_n(\cdot)$ denotes the `$n$'th Legendre polynomial.
For this given choice of initial data, and domain evolution, an exact solution of the Stefan problem is known and constructed in \S\ref{subappendix: varying radius}.

\begin{figure}[ht]
\centering
\makebox[\textwidth][c]{
\begin{subfigure}[b]{.4\textwidth}
\begin{tikzpicture}
\begin{axis}[
    width=\textwidth,
    height=\textwidth,
    scale only axis,
    xmode=log,
    ymode=log,
    xlabel={$h$},
    legend pos=south east,
    grid=both,
    log basis x=2,
    log basis y=10,
]
\addplot[
    mark=*,
    blue,
] coordinates {
    (0.139,0.00024475)
    (0.286,0.00098834)
    (0.5245,0.0036786)
    (1.0249,0.012029)
};
\addlegendentry{$\|u-u_h^\ell\|_{L^2_{L^2}}$};
\addplot[
    mark=diamond*,
    red,
] coordinates {
    (0.139,0.21693 )
    (0.286,0.41737)
    (0.5245,0.78022)
    (1.0249,1.3571)
};
\addlegendentry{$\|e-e_h^\ell\|_{L^\infty_{H^{-1}}}$};
\addplot[
    black, thick
] coordinates {
    (0.286,0.286)
    (0.572,0.572)
};

\addplot[
    black, thick
] coordinates {
    (0.286,0.002)
    (0.572,0.008)
};
\end{axis}
\node at (3.3,4.5) {$\mathcal{O}(h)$};
\node at (3.3,2.6) {$\mathcal{O}(h^2)$};
\end{tikzpicture}
\caption{Error plots for \eqref{eqn: matrix vector form quadrature} on an expanding sphere.}
\end{subfigure}
\hspace{2cm}%
\begin{subfigure}[b]{.4\textwidth}
\begin{tikzpicture}
\begin{axis}[
    width=\textwidth,
    height=\textwidth,
    scale only axis,
    xmode=log,
    ymode=log,
    xlabel={$h$},
    legend pos=south east,
    grid=both,
    log basis x=2,
    log basis y=10,
]
\addplot[
    mark=*,
    blue,
] coordinates {
    (0.139,0.00075059)
    (0.286,0.0020292)
    (0.5245,0.006165)
    (1.0249,0.015497)
};
\addlegendentry{$\|u-u_h^\ell\|_{L^2_{L^2}}$};
\addplot[
    mark=diamond*,
    red,
] coordinates {
    (0.139,0.21695)
    (0.286,0.41737)
    (0.5245,0.78018)
    (1.0249,1.3563)
};
\addlegendentry{$\|e-e_h^\ell\|_{L^\infty_{H^{-1}}}$};
\addplot[
    black, thick
] coordinates {
    (0.286,0.286)
    (0.572,0.572)
};

\addplot[
    black, thick
] coordinates {
    (0.286,0.004)
    (0.572,0.016)
};
\end{axis}
\node at (3.3,4.4) {$\mathcal{O}(h)$};
\node at (3.3,2.6) {$\mathcal{O}(h^2)$};
\end{tikzpicture}
\caption{Error plots for \eqref{eqn: matrix vector form exact} on an expanding sphere.}
\end{subfigure}
}
\caption{Error plots for the temperature (blue) and enthalpy (red), for both schemes~\eqref{eqn: matrix vector form quadrature} and~\eqref{eqn: matrix vector form exact}, on an expanding sphere as detailed in Section \ref{sec:expSphere}.
Notice that in both cases the error is better than the $\mathcal{O}(\sqrt{h})$ error predicted by Theorem~\ref{thm: error theorem}.
\label{fig: EOC_expanding_sphere}}
\end{figure}
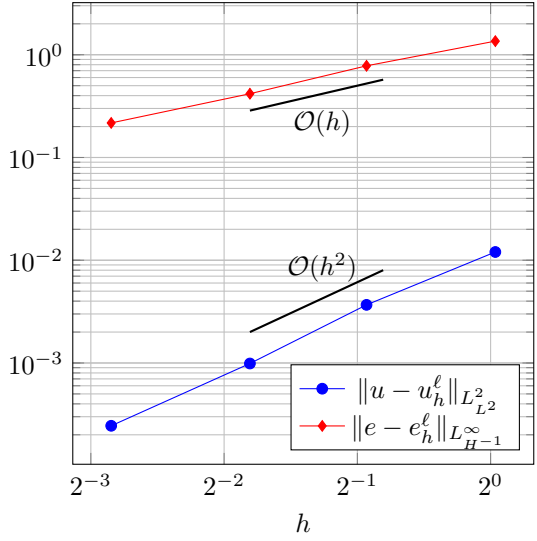
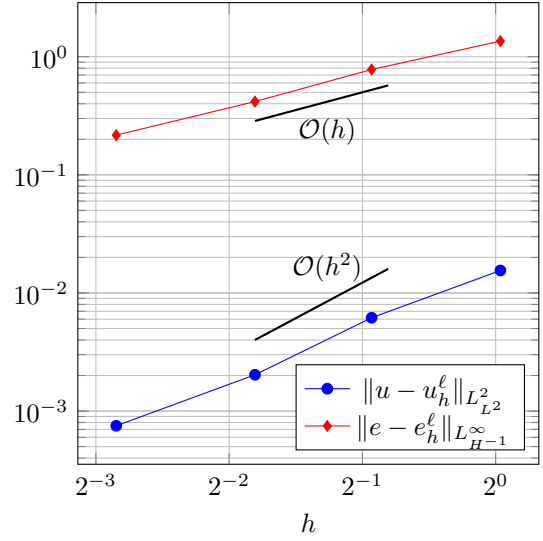

\subsubsection{Rotating sphere}\label{sec:rotSphere}

Here we demonstrate experimental order of convergence results using the rotating sphere solution constructed in \S\ref{subappendix: rotating sphere}.
We consider the evolution of the unit sphere under the map $\Phi$ given by
\[\Phi(\mbf{x};t) = \begin{pmatrix} 1 & 0 & 0\\
0 & \cos(t) & \sin(t)\\
0 & -\sin(t) & \cos(t)
\end{pmatrix} \mbf{x}.\]
Here we choose our initial data to be given by the Lagrange interpolant of $e_0$, where
\[ e_0(x,y,z) = \begin{cases}
    P_1(y) + \frac{1}{2}P_3(y), & y \leq 0,\\
    1 + P_1(y) + \frac{1}{2}P_3(y), & y > 0,
\end{cases} \]
where $P_n(\cdot)$ denotes the `$n$'th Legendre polynomial.

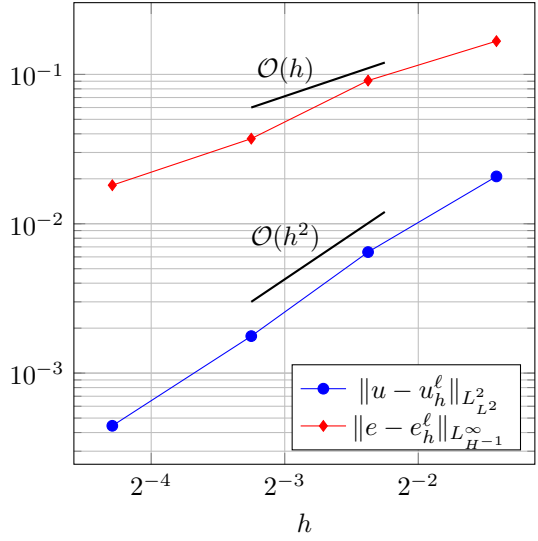
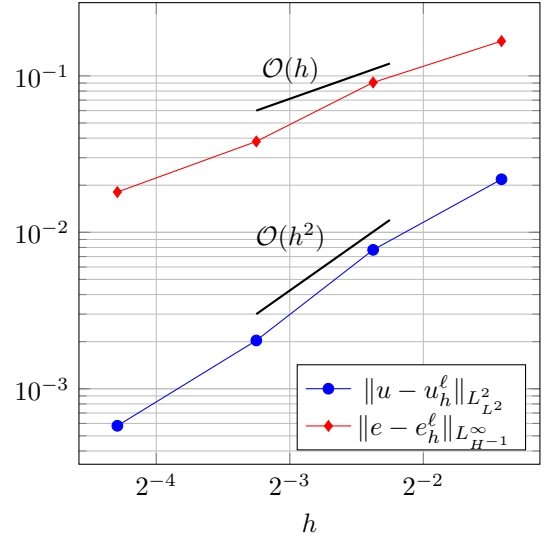
\begin{figure}[ht]
\centering
\makebox[\textwidth][c]{
\begin{subfigure}[b]{.4\textwidth}
\begin{tikzpicture}
\begin{axis}[
    width=\textwidth,
    height=\textwidth,
    scale only axis,
    xmode=log,
    ymode=log,
    xlabel={$h$},
    legend pos=south east,
    grid=both,
    log basis x=2,
    log basis y=10,
]
\addplot[
    mark=*,
    blue,
] coordinates {
    (0.0511,0.00044387)
    (0.1051,0.0017688)
    (0.1927,0.0064622)
    (0.3752,0.020743)
};
\addlegendentry{$\|u-u_h^\ell\|_{L^2_{L^2}}$};
\addplot[
    mark=diamond*,
    red,
] coordinates {
    (0.0511,0.018117)
    (0.1051,0.03713)
    (0.1927,0.090914)
    (0.3752,0.16701)
};
\addlegendentry{$\|e-e_h^\ell\|_{L^\infty_{H^{-1}}}$};
\addplot[
    black, thick
] coordinates {
    (0.1051,0.06)
    (0.2102,0.12)
};

\addplot[
    black, thick
] coordinates {
    (0.1051,0.003)
    (0.2102,0.012)
};
\end{axis}
\node at (2.8,5.2) {$\mathcal{O}(h)$};
\node at (2.8,3) {$\mathcal{O}(h^2)$};
\end{tikzpicture}
\caption{Error plots for \eqref{eqn: matrix vector form quadrature} on a rotating sphere as detailed in Section \ref{sec:rotSphere}.}
\end{subfigure}
\hspace{2cm}%
\begin{subfigure}[b]{.4\textwidth}
\begin{tikzpicture}
\begin{axis}[
    width=\textwidth,
    height=\textwidth,
    scale only axis,
    xmode=log,
    ymode=log,
    xlabel={$h$},
    legend pos=south east,
    grid=both,
    log basis x=2,
    log basis y=10,
]
\addplot[
    mark=*,
    blue,
] coordinates {
    (0.0511,0.00057966)
    (0.1051,0.0020351)
    (0.1927,0.0077256)
    (0.3752,0.021837)
};
\addlegendentry{$\|u-u_h^\ell\|_{L^2_{L^2}}$};
\addplot[
    mark=diamond*,
    red,
] coordinates {
    (0.0511,0.018092)
    (0.1051,0.038086)
    (0.1927,0.090757)
    (0.3752,0.16688)
};
\addlegendentry{$\|e-e_h^\ell\|_{L^\infty_{H^{-1}}}$};
\addplot[
    black, thick
] coordinates {
    (0.1051,0.06)
    (0.2102,0.12)
};

\addplot[
    black, thick
] coordinates {
    (0.1051,0.003)
    (0.2102,0.012)
};
\end{axis}
\node at (2.8,5.2) {$\mathcal{O}(h)$};
\node at (2.8,3) {$\mathcal{O}(h^2)$};
\end{tikzpicture}
\caption{Error plots for \eqref{eqn: matrix vector form exact} on a rotating sphere.}
\end{subfigure}
}
\caption{Error plots for the temperature (blue) and enthalpy (red), for both methods~\eqref{eqn: matrix vector form quadrature} and~\eqref{eqn: matrix vector form exact}, on a rotating sphere.
Notice that in both cases the error is better than the $\mathcal{O}(\sqrt{h})$ error predicted by Theorem~\ref{thm: error theorem}. \label{fig: EOC_rotating_sphere}}
\end{figure}

\subsection*{\textbf{Acknowledgments}}
The authors would like to thank Vanessa Styles and James Van Yperen for their comments on an early draft of this manuscript.
TS is supported by the UK Engineering and Physical Sciences Research Council (Grant number: EP/Z535138/1).
PJH, TS, and CV are all supported by a UK Engineering and Physical Sciences Research Council Mathematical Sciences Small Grant.

\subsection*{\textbf{Conflict of interest}}
There are no conflicts of interest to declare.

\subsection*{\textbf{Data availability}}
The code used for the numerical examples in this paper is available upon reasonable request to the authors.

\appendix
\section{Exact solutions for the Stefan problem on an evolving sphere}
\label{appendix: exact solution}
In this appendix we construct some exact solutions for the Stefan problem on an evolving sphere.
We consider two kinds of evolution: uniform expansion/dilation in the normal direction, and tangential motion given by a uniform rotation.
Note that this latter case corresponds to a Stefan problem with advection posed on a stationary surface.
We also refer the reader to \cite{Garcke2024CrystalGrowth,Ratz2016Benchmark} wherein exact solutions to surface free boundary problems (namely the Mullins--Sekerka problem) are constructed for use as a benchmark.
\subsection{A sphere with varying radius}
\label{subappendix: varying radius}
Here we manufacture a solution for the the strong formulation of the Stefan problem,~\eqref{eqn: strong stefan}, on a shrinking sphere, with a fixed interface at $z=0$, to be used in the experimental order of convergence calculations in Section~\ref{section: numerics}.
Notice that although the calculations in~\cite[Remark 2.12]{Alphonse2015Stefan} do not consider an external heat source, given that our manufactured solution will be chosen such that there are no mushy regions we will be able to verify that it is also a solution of the enthalpy formulation.
We shall consider a ball of radius $\rho(t)$, centred at the origin, which we shall denote as $S^2(\rho(t))$
To construct our solution we will firstly solve the heat equation~\eqref{eqn: strong stefan 1} on the upper hemisphere $S^2(\rho(t)) \cap \{z \geq 0\}$ subject to homogeneous Dirichlet boundary conditions on $S^2(\rho(t)) \cap \{z=0\}$.
For this we pullback~\eqref{eqn: strong stefan 1} onto the unit sphere $S^2(1)$, noting that as we assume evolution is exclusively in the normal direction one finds
\[ \grado \cdot \mbf{V} = HV_N = \frac{2\rho'(t)}{\rho(t)}, \]
where $V_N = \rho'(t)$ is the normal velocity of the sphere and $H = \frac{2}{\rho(t)}$ is (twice) the mean curvature.
By using the definition of the material derivative, and the pullback equation for the Laplace--Beltrami operator~\cite[Lemma 3.4]{Church2020DomainMapping}, one can readily observe that the pullback $\widetilde{u}$ solves
\[ \frac{\partial \widetilde{u}}{\partial t} + 2(\widetilde{u} + 1)\frac{\rho'(t)}{\rho(t)} - \frac{1}{\rho(t)^2} \Delta_{S^2} \widetilde{u} = \widetilde{f}, \]
where $\widetilde{f}$ is the pullback of $f$ onto $S^2(1)$, and $\Delta_{S^2}$ is the Laplace--Beltrami operator on $S^2(1)$.

We now choose $\widetilde{u}$ as
\[ \widetilde{u}(x,y,z;t) = e^{-2t}P_{1}(z) + \frac{1}{2}e^{-12t}P_{3}(z),  \]
where $P_n(\cdot)$ denotes the `$n$'th Legendre polynomial.
It is well-known (see for instance~\cite{Muller1966SphericalHarmonics}) that the eigenfunctions of the Laplace--Beltrami operator on $S^2(1)$ are the spherical harmonics, which includes the family $\{P_n(z)\}_{n \in \mbb{N} \cup \{0\}}$, where one finds that
\[ \Delta_{S^2} P_n(z) = -n(n+1) P_n(z). \]
Moreover, by choosing only odd degree spherical harmonics we also satisfy the Dirichlet boundary condition, owing to the property that $P_{2n+1}(0) = 0$ for all $n \in \mbb{N}$.
As such, we now take $\widetilde{f}$ on the upper hemisphere $S^2(1) \cap \{ z \geq 0\}$ to be given by
\begin{align*}
\widetilde{f}(x,y,z;t) &= (-2e^{-2t}P_{1}(z)-6 e^{-12t}P_{3}(z))\left(1-\frac{1}{\rho(t)^2}\right)+ (2+ 2e^{-2t}P_{1}(z) + e^{-12t}P_{3}(z))\frac{\rho'(t)}{\rho(t)}\\
&=: \widetilde{f}_+(z;t).
\end{align*}

Repeating essentially the same calculations for~\eqref{eqn: strong stefan 2} we find that
\begin{align*}
    \widetilde{f}(x,y,z;t) &= (-2e^{-2t}P_{1}(z)-6 e^{-12t}P_{3}(z))\left(1-\frac{1}{\rho(t)^2}\right)+ (2e^{-2t}P_{1}(z) + e^{-12t}P_{3}(z))\frac{\rho'(t)}{\rho(t)}\\
    & =: \widetilde{f}_-(z;t),
\end{align*}
on the lower hemisphere $S^2(1) \cap \{ z < 0\}$.
All that remains is to push-forward our functions onto $S^2(\rho(t))$, for which we find that
\[u(x,y,z;t) = e^{-2t}P_{1}\left(\frac{z}{\rho(t)}\right) + \frac{1}{2}e^{-12t}P_{3}\left(\frac{z}{\rho(t)}\right),\]
and
\begin{align}
    f(x,y,z;t) = \begin{cases}
    \widetilde{f}_+\left( \frac{z}{\rho(t)}; t \right), & z \geq 0,\\
    \widetilde{f}_-\left( \frac{z}{\rho(t)}; t \right), & z < 0.
\end{cases}
\label{eqn: manufactured solution rhs}
\end{align}
We note that the values $f$ takes on the interface $S^2(\rho(t)) \cap \{z  = 0 \}$ are seemingly irrelevant as the PDEs are defined away from this interface.
Moreover, since the Laplace--Beltrami operator is invariant under isometries, we may choose the interface to be any great circle rather than specifically $\{z = 0\}$.
We shall do this in our examples, where we instead choose the interface to be $\{y = 0\}$, so that the interface is not fitted to our mesh.
In our numerical experiments in Section~\ref{section: numerics} we shall use this manufactured solution with $\rho(t) = e^{2t}$.

\subsection{A rotating sphere}
\label{subappendix: rotating sphere}
As a second benchmark in Section~\ref{section: numerics} we shall consider $\Omega(t)$ to be a rotating unit sphere.
In this case the evolution of the sphere is given by
\[ \mbf{V}(\mbf{x};t) = A \mbf{x}, \quad \forall \mbf{x} \in S^2(1),\]
where $\mbf{A} \in \mbb{R}^{3 \times 3}$ is some given antisymmetric matrix.
It is a straightforward calculation to verify that $\grado \cdot \mbf{V} = 0$, and hence the temperature, $u$, solves the same PDE on both sides of the interface.
We will solve the PDE by pulling back onto $S^2(1)$, as we did above, where we note that
\[ \Phi_{-t}(\lapo u) = \Delta_{S^2} (\Phi_{-t}u), \]
since $\Phi(t) : S^2(1) \rightarrow \Omega(t)$ is an isometry.
Thus, the pullback of~\eqref{eqn: strong stefan 1} is now
\[\frac{\partial \widetilde{u}}{\partial t} - \Delta_{S^2} \widetilde{u} = 0,\]
where $\widetilde{u} = \Phi_{-t}u$ is the pullback of $u$ onto $S^2(1)$.
Again by noting that the functions $P_n(z)$ are eigenfunctions of the Laplace--Beltrami operator on $S^2(1)$ we find that the above PDE is solved by
\[\widetilde{u}(x,y,z;t) = e^{-2t} P_1(z) + \frac{1}{2} e^{-12t} P_3(z),\]
which also satisfies the homogeneous Dirichlet condition on $z = 0$.
Moreover, this is a solution on both sides of the interface, i.e.~$S^2(1) \cap \{ \widetilde{u} < 0\}$ and $S^2(1) \cap \{ \widetilde{u} > 0\}$.
Thus a strong solution to the Stefan problem on $\Omega(t)$ is given by
\[ u(x,y,z;t) = e^{-2t} P_1(\Phi_t z) + \frac{1}{2} e^{-12t} P_3(\Phi_t z),\]
where the free boundary is the curve $\Gamma(t) = S^2(1) \cap \{\Phi_t z = 0\}$.
As above, we may we may choose the initial interface to be any great circle rather than $\{z = 0\}$.

\bibliographystyle{acm}
\bibliography{stefanBib.bib}

\end{document}